\documentclass[11pt]{amsart}

\usepackage{tikz}
\usepackage[all]{xy}
\usepackage{amsmath,amsfonts,amsthm,amssymb,amscd}
\usepackage{upgreek}
\usepackage[margin=2.5cm]{geometry}
\usepackage{latexsym}
\usepackage{hyperref}
\usepackage{graphicx}
\usepackage{float}

\usepackage{color}

\newtheorem{Thm}{Theorem}[section]
\newtheorem{Cor}[Thm]{Corollary}
\newtheorem{Conj}[Thm]{Conjecture}
\newtheorem{Prop}[Thm]{Proposition}

\newtheorem{Lem}[Thm]{Lemma}

\theoremstyle{definition}
\newtheorem{Def}[Thm]{Definition}
\newtheorem{Not}[Thm]{Notation}

\newtheorem{Rem}[Thm]{Remark}

\theoremstyle{remark}

\numberwithin{equation}{section}

\newcommand{\Aut}{\operatorname{Aut}}

\newcommand{\Hom}{\operatorname{Hom}}

\newcommand{\Syl}{\operatorname{Syl}}

\newcommand{\Out}{\operatorname{Out}}

\renewcommand{\dim}{\operatorname{dim}}
\newcommand{\Inn}{\operatorname{Inn}}

\newcommand{\Coind}{\operatorname{Coind}}

\def \bb {\mathfrak{b}}
\def \qq {\mathfrak{q}}

\newcommand{\res}{\operatorname{res}}
\newcommand{\tr}{\operatorname{tr}}

\newcommand{\Spin}{\operatorname{Spin}}

\newcommand{\Sol}{\operatorname{Sol}}

\newcommand{\GL}{\operatorname{GL}}
\newcommand{\SL}{\operatorname{SL}}

\newcommand{\Irr}{\operatorname{Irr}}
\newcommand{\SO}{\operatorname{SO}}

\renewcommand{\epsilon}{\varepsilon}
\renewcommand{\bar}{\overline}
\renewcommand{\hat}{\widehat}
\renewcommand{\leq}{\leqslant}
\renewcommand{\geq}{\geqslant}

\renewcommand{\phi}{\varphi}

\newcommand{\C}{\mathcal{C}}

\newcommand{\A}{\mathcal{A}}

\newcommand{\D}{\mathcal{D}}

\newcommand{\F}{\mathcal{F}}

\renewcommand{\H}{\mathcal{H}}

\newcommand{\K}{\mathcal{K}}

\newcommand{\Q}{\mathcal{Q}}

\newcommand{\Ab}{\mathcal{A}b}

\newcommand{\ZZ}{\mathbb{Z}}
\newcommand{\CC}{\mathbb{C}}
\newcommand{\FF}{\mathbb{F}}

\newcommand{\gen}[1]{\langle #1 \rangle}
\newcommand{\norm}{\unlhd}
\newcommand{\Lim}[1]{\lim_{#1}}

\newcommand{\smat}[1]{\left[\begin{smallmatrix}#1\end{smallmatrix}\right]}
\newcommand{\diag}{\operatorname{diag}}

\begin{document}
\title{Weights in a Benson-Solomon block}

\author{Justin Lynd}
\address{Department of Mathematics\\University of Louisiana at Lafayette\\Lafayette, LA 70504}
\email{lynd@louisiana.edu}

\author{Jason Semeraro}
\address{Heilbronn Institute for Mathematical Research, Department of
Mathematics, University of Leicester,  United Kingdom}
\email{jpgs1@leicester.ac.uk}

\thanks{
\begin{minipage}[s]{0.1\textwidth}
\includegraphics[scale=0.23]{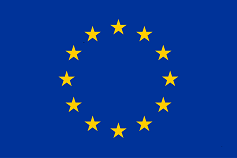}
\end{minipage}
\begin{minipage}[s]{0.9\textwidth}
This project has received funding from the European Union's Horizon 2020
research and innovation programme under the Marie Sk{\l}odowska-Curie grant
agreement No. 707758.
\end{minipage}
}

\begin{abstract} 
To each pair consisting of a saturated fusion system over a $p$-group together
with a compatible family of K\"{u}lshammer-Puig cohomology classes, one can
count weights in a hypothetical block algebra arising from these data. When the
pair arises from a genuine block of a finite group algebra in characteristic
$p$, the number of conjugacy classes of weights is supposed to be the number of
simple modules in the block. We show that there is unique such pair associated
with each Benson-Solomon exotic fusion system, and that the number of weights
in a hypothetical Benson-Solomon block is $12$, independently of the field of
definition.  This is carried out in part by listing explicitly up to conjugacy
all centric radical subgroups and their outer automorphism groups in these
systems.  
\end{abstract}



\maketitle

\section{Introduction}
Let $k$ be an algebraically closed field of characteristic $p > 0$, and let $G$
be a finite group. Associated to each block $b$ of $kG$, there is a saturated
fusion system $\F = \F_S(b)$ over the defect group $S$ of the block in which
the morphisms between subgroups are given by conjugation by elements of $G$
preserving the corresponding Brauer pairs \cite{AschbacherKessarOliver2011,
CravenTheory}.  Several questions in the modular representation theory of
finite groups concern the connection between representation theoretic
properties of $kGb$ and the category $\F$. However, it is known that for many
purposes $\F$ does not, in general, retain enough information about $kGb$-{\rm
mod}. For example, it does not determine the number of simple modules in $b$,
in part because it retains too little of the $p'$-structure of $p$-local
subgroups.

On the other hand, the block $b$ also determines a family of degree 2
cohomology classes $\alpha_Q \in H^2(\Aut_\F(Q), k^\times)$, for $Q \in \F^c$
an $\F$-centric subgroup, by work of K\"ulshammer and Puig (see
\cite[IV.5.5]{AschbacherKessarOliver2011}.) This family is expected to supply
the missing information away from the prime $p$.  The K\"ulshammer-Puig classes
are compatible in the sense that, by \cite[Theorem~8.14.5]{LinckelmannBT2},
they determine an element
\[
\alpha \in \lim_{[S(\F^c)]} \A_\F^2
\]
where $[S(\F^c)]$ is the poset of $\F$-isomorphism classes of chains
$\sigma=(X_0 < X_1 < \cdots < X_n)$ of $\F$-centric subgroups, and $\A_{\F}^2$
is the covariant functor which sends a chain $\sigma$ to $H^2(\Aut_\F(\sigma),
k^\times)$. Here, $\Aut_\F(\sigma) \leq \Aut_\F(X_n)$ is the group of
automorphisms in $\F$ of $X_n$ preserving all members $X_i$ of the chain.  For
example, if $b$ is the principal block of $kG$ then $\alpha$ is always the
trivial class \cite[IV.5.32]{AschbacherKessarOliver2011}. 

Thus, by a \textit{K\"{u}lshammer-Puig pair}, we mean a pair $(\F,\alpha)$
where $\F$ is a saturated fusion system on a $p$-group $S$ and $\alpha$ is an
element of $\lim_{[S(\F^c)]} \A_\F^2$.  Given such a pair $(\F,\alpha)$ arising
from a block $b$, the quantity
\begin{equation*}
\mathbf{w}(\F,\alpha):=\sum_{Q \in \F^{cr}/\F} z(k_{\alpha_Q}\Out_\F(Q)), 
\end{equation*} 
counts the number of $kGb$-weights. Here, $k_{\alpha_Q}\Out_\F(Q)$ is the
algebra obtained from the group algebra $k\Out_\F(Q)$ by twisting with
$\alpha_Q$ \cite[IV.5.36]{AschbacherKessarOliver2011}, $z(-)$ denotes the
number projective simple modules, and the sum is taken over a set of
representatives for the conjugacy classes of $\F$-centric and $\F$-radical
subgroups.  Thus, Alperin's Weight Conjecture says that $\mathbf{w}(\F,\alpha)$
is the number of simple $kGb$-modules
\cite[IV.5.46]{AschbacherKessarOliver2011}.

There is always a natural map $H^{2}(\F^c, k^{\times}) \to \Lim{[S(\F^c)]}
\A_\F^2$, and the \emph{gluing problem} asks whether this map is surjective
(see \cite{Linckelmann2009b} and \cite{Libman2011} for further details).
Linckelmann has shown that Alperin's conjecture has a structural reformulation
in terms of algebras constructed from $p$-local finite groups provided the
gluing problem always has a solution \cite{Linckelmann2004}.  However, while
the weight conjecture has relevance for actual blocks only, the gluing problem
is a question about the K\"{u}lshammer-Puig pair itself and can be considered
(1) when $\F$ is the fusion system of a block, but of no block with the
specified compatible family $\alpha$, and (2) when $\F$ is the fusion system of
no block at all.  Thus, we are interested in investigating such pairs
disembodied from an actual block as a way of gauging the degree to which
certain questions, and potential answers to those questions, are $p$-locally
determined. A direct study of K\"{u}lshammer-Puig pairs might reveal, for
example, that there is an exotic pair as in (1) or (2) that does not satisfy
the gluing problem. At this stage, such a possibility seems unlikely. On the
other hand, and conversely, we would be very interested in a structural
explanation why the gluing problem should hold in general, and it seems
reasonable to expect that such an explanation would apply to all such pairs,
exotic or not.

In this paper we consider K\"{u}lshammer-Puig pairs associated with the exotic
family $\Sol(q)$ of Benson-Solomon $2$-fusion systems
\cite{AschbacherChermak2010, LeviOliver2002}. These systems are defined for any
odd prime power, but $\Sol(q)$ and $\Sol(q')$ are isomorphic as fusion systems
if and only if $v_2(q^{2}-1) = v_2(q'^{2}-1)$, where $v_2$ is the $2$-adic
valuation.  A Benson-Solomon system is known not to be the fusion system of any
genuine block. This is a result of Kessar for the smallest such system
\cite{Kessar2006}, while Craven extended Kessar's proof to the general case in
\cite[Theorem 9.34]{CravenTheory}.  Our first theorem determines the possible
K\"{u}lshammer-Puig classes that these fusion systems support.

\begin{Thm}\label{t:solblocks}
Let $\F = \Sol(q)$. Then 
\[
\lim_{[S(\F^{c})]} \A_\F^2 \cong \lim_{[S(\F^{cr})]} \A_\F^2 = 0.
\]
That is, each Benson-Solomon system supports a unique K\"{u}lshammer-Puig pair.
\end{Thm}

Theorem~\ref{t:solblocks} is shown by explicitly computing the $\F$-conjugacy
classes of centric radical subgroups along with their outer automorphism groups
in $\F$. The results of \cite[Section 10]{AschbacherChermak2010} go a long way
towards accomplishing such a task, but more details are required for the
present applications. In Section~\ref{s:solq} we refine the results of
\cite{AschbacherChermak2010} to prove the following.

\begin{Thm}\label{t:main}
Let $\F = \Sol(q)$. Representatives for the $\F$-conjugacy classes of
$\F$-centric radical subgroups, together with their $\F$-outer automorphism
groups, are listed in Tables \ref{t:1} and \ref{t:2}. 
\end{Thm}

\begin{Thm}\label{t:solq0}
The number of weights in the unique pair of Theorem~\ref{t:solblocks} is
\[
\mathbf{w}(\Sol(q),0) = 12,
\]
independently of $q$. 
\end{Thm}

We prove this result in Section \ref{s:blocks} by explicitly computing the
quantity $z(k\Out_\F(Q))$ for each of the groups $Q$ appearing in Tables
\ref{t:1} and \ref{t:2} of Theorem~\ref{t:main}. 

Beyond the weight conjecture, and assuming its validity, we have in mind other
counting questions that can be considered for K\"{u}lshammer-Puig pairs without
reference to a group or a block. For example, Malle and Robinson recently
conjectured that if $b$ is a $p$-block associated to a finite group $G$ then
the number of simple $kG$-modules in $b$ should be bounded by $p^{s(S)}$, where
$S$ is a defect group of $b$ and $s(S)$ denotes the sectional rank of $S$,
namely the largest rank of an elementary abelian section
\cite{MalleRobinson2017}. Moreover, they verified their conjecture in a large
number of cases where the weight conjecture holds. In Lemma~\ref{L:sectrank},
we observe that the sectional rank of $S$ is $6$, and so the following conjecture,
which was suggested to us by Kessar and Linckelmann, also holds easily for
$\Sol(q)$. 

\begin{Conj}\label{c:kesslinck}
Let $(\F,\alpha)$ be a K\"ulshammer-Puig pair, where $\F$ is a saturated fusion
system on $S$. Then $\mathbf{w}(\F,\alpha) \le p^{s(S)}$.
\end{Conj}

This conjecture is just one small example in a host of other conjectures
which are certain purely local analogues of the various local-to-global
conjectures in the modular representation theory of finite groups. The local
conjectures by their nature do not discriminate between realizable and exotic
K\"ulshammer-Puig pairs. They are discussed more fully in a sequel to this
paper \cite{kessar2018weight}. 

\subsection*{Outline, and notation for the tables}
After recalling certain initial results about fusion systems and the $2$-local
structure of $\SL_2(q)$, we set up in Section~\ref{s:solq} notation for working
in the Benson-Solomon systems and identify the important subgroups of the
Aschbacher-Chermak free amalgamated product which realizes the systems.
Section~\ref{ss:abovetorus} provides an initial classification of some centric
radical subgroups, namely the centric radical subgroups lying above the
$2$-torsion in a maximal torus.  

Section~\ref{s:solqcrel} contains the proof of Theorem~\ref{t:main}, where the
smallest Benson-Solomon system is handled separately (Subsection~\ref{s:l=0})
from the larger ones (Subsection~\ref{s:kcr}). The results are summarized in
Tables~\ref{t:1} and \ref{t:2}. Those tables give a list of subgroups whose
notation was fixed previously in Notation~\ref{N:Knotation1},
Notation~\ref{N:Knotation2}, Subsection~\ref{ss:standardseq},
\eqref{e:Qforl=0}, or Notation~\ref{N:R**}.

Theorem~\ref{t:solq0} is proved in Section~\ref{s:blocks}.  Finally, in
Section~\ref{s:KP}, we compute the Schur multipliers of the outer automorphism
groups to give a proof of Theorem~\ref{t:solblocks}. 

\subsection*{Acknowledgements}
We thank Markus Linckelmann and Radha Kessar for encouragement, Geoffrey
Robinson for comments on an earlier version, and Dave Benson and Ian Leary for
helpful conversations concerning Section~\ref{s:KP}.   We express our gratitude
to the referee whose thorough reading and many suggestions resulted in numerous
improvements to the paper. The first named author thanks the European
Commission for funding through a Marie Curie Fellowship, without which this
work would have not materialized.

\section{The Benson-Solomon fusion systems}\label{s:solq}

\subsection{Fusion system preliminaries}\label{ss:prelim}

Throughout this paper, our group-theoretic nomenclature is standard and follows
\cite{Wilson2009}, and we are usually consistent with the fusion-theoretic
terminology and notation of \cite{AschbacherKessarOliver2011}. One exception to
this is that we use exponential notation for images of subgroups and elements
under a morphism in a fusion system, as described below.

A \emph{fusion system} on a finite $p$-group $S$ is a category with objects the
subgroups of $S$, and with morphisms injective group homomorphisms subject to
two weak axioms. The standard example of a fusion system is that of a finite
group $G$ with Sylow $p$-subgroup $S$, where the morphisms are the conjugation
homomorphisms between subgroups of $S$ induced by elements of the group $G$,
and which is denoted $\F_S(G)$. Due to the validity of Sylow's Theorem in $G$
and its $p$-local subgroups, the standard example satisfies two additional
\emph{saturation axioms}, the Sylow and Extension axioms
\cite[Definition~1.2]{BrotoLeviOliver2003}.  All fusion systems in this paper
are assumed to be (or known already to be) saturated unless otherwise stated,
and we will sometimes drop that adjective and speak simply of a fusion system
when there is no cause for confusion.

For this subsection, we fix a saturated fusion system $\F$ over the $p$-group
$S$.  By analogy with the standard example, two subgroups of $S$ are said to be
$\F$-\textit{conjugate} if they are isomorphic in the category $\F$.  For a
morphism $\phi\colon P \to Q$ in $\F$, we write $P^\phi$ for the image of
$\phi$. Similarly, $x^\phi$ denotes the image of an element $x$ under a
morphism whose domain contains $x$.

\begin{Def}\label{d:centrad} Fix a subgroup $P \leq S$. We say that $P$ is
\begin{itemize} \item[(a)] \textit{fully $\F$-normalized} if $|N_S(P)| \geq
|N_S(Q)|$ whenever $Q$ is $\F$-conjugate to $P$, \item[(b)]
\textit{$\F$-centric} if  $C_S(Q)=Z(Q)$ for each $\F$-conjugate $Q$ of $P$,
\item[(c)] \textit{$\F$-radical} if $O_p(\Out_\F(P))=1$, \item[(d)]
\textit{$\F$-centric radical} if it is both $\F$-centric and $\F$-radical, and
\item[(e)] \textit{weakly $\F$-closed} if $P$ is the only $\F$-conjugate of $P$,
\item[(f)] \textit{strongly $\F$-closed} if each $\F$-conjugate of a subgroup
of $P$ is contained in $P$. 
\end{itemize} Denote by $\F^c$, $\F^r$, and $\F^{cr}$ the collection
of $\F$-centric, $\F$-radical, and $\F$-centric radical subgroups of $S$,
respectively.  \end{Def}

The collections $\F^c$, $\F^{r}$ and $\F^{cr}$ are all closed under
$\F$-conjugacy. Also, the $\F$-centric subgroups are closed under passing to
overgroups.

\begin{Rem} \label{r:radical} Let $G$ be a finite group with Sylow $p$-subgroup
$S$.  A $p$-subgroup $P$ of $G$ is said to be $p$-\emph{radical} in $G$ if
$O_p(N_G(P)/P) = 1$. By contrast, a subgroup $P$ is $\F_S(G)$-radical if and
only if $O_p(N_G(P)/PC_G(P)) = 1$.  The collection of $p$-radical subgroups of
$G$ contained in $S$ does \emph{not} coincide, in general, with the collection
of $\F_S(G)$-radical subgroups. 

For example, let $p = 3$ and $G = G_1 \times G_2$ with $G_i \cong D_6$.  The
subgroup $P = S \cap G_1$ has order $3$ with $N_G(P)/P \cong C_2 \times D_6$,
so $P$ is not $3$-radical in $G$.  However, $\Out_{\F_S(G)}(P) = N_G(P)/PC_G(P)
\cong C_2$, so $P$ is $\F_S(G)$-radical. Conversely, take $p = 2$ but instead
$G = D_{24}$, and $P$ of order $4$ in the cyclic maximal subgroup. Then
$\Out_{\F_S(G)}(P) \cong C_2$ so $P$ is not $\F_S(G)$-radical, but $N_G(P)/P
\cong D_{6}$ so $P$ is $2$-radical in $G$. 

This distinction is important in Lemma~\ref{l:abovewc} below, where both
concepts appear simultaneously. It is also relevant in Chevalley groups $G =
G(q)$ with $q$ odd which have an element in the Weyl group inverting a split
maximal torus. When such a torus has nontrivial odd order normal subgroup
(often the case), a Sylow $2$-subgroup $T$ of such a torus is $2$-radical in
$G$ but not radical in $\F_S(G)$, where $S$ is a Sylow $2$-subgroup of $G(q)$
containing $T$.  This situation occurs for example when $G(q) = \Spin_7(q)$,
$q$ odd, $q \neq 3, 5$.  \end{Rem}

\begin{Def} Fix a subgroup $P \leq S$.  \begin{itemize} \item[(a)] The
\textit{normalizer} $N_\F(P)$ of $P$ is the fusion system on $N_S(P)$
consisting of those morphisms $\phi\colon Q \to R$ in $\F$ for which there
exists an extension $\tilde{\phi}\colon PQ \to PR$ of $\phi$ in $\F$ such that
$P^{\tilde{\phi}} = P$.  \item[(b)] The \textit{centralizer} $C_\F(P)$ of $P$
is the fusion system on $C_S(P)$ consisting of those morphisms $\phi\colon Q
\to R$ in $\F$ for which there exists an extension $\tilde{\phi}\colon PQ \to
PR$ of $\phi$ in $\F$ such that the restriction $\tilde{\phi}|_P$ is the
identity on $P$.  \item[(c)] The subgroup $P\leq S$ is \emph{normal} in $\F$ if
$\F = N_\F(P)$.
\item[(d)] $\F$ is \emph{constrained} if $\F$ has a centric normal
subgroup.
\end{itemize} \end{Def}

These centralizer and normalizer fusion systems are not always saturated, but
they are both saturated provided $P$ is fully $\F$-normalized.  

\begin{Lem}\label{l:fusion-system-of-the-center} If $P$ is $\F$-centric, then
$C_\F(P) = \F_{Z(P)}(Z(P))$.  \end{Lem} \begin{proof} Assume that $P$ is
$\F$-centric.  The centralizer system $C_\F(P)$ is a fusion system over the
abelian group $C_S(P) = Z(P)$, and $Z(P)$ is normal in $C_\F(P)$ from the
definitions. As each morphism between subgroups of $Z(P)$ in $C_\F(P)$ extends
to act as the identity on $P$, each such morphism is an identity map.
\end{proof}

\begin{Lem}\label{l:normcentrad} Suppose that $P \leq S$ is normal in $\F$.
Then $P$ is contained in every $\F$-centric radical subgroup.  \end{Lem}
\begin{proof} Let $Q \in \F^{cr}$. Then $\Aut_{PQ}(Q)$ is normal in
$\Aut_\F(Q)$, and so $\Aut_{PQ}(Q) \leq \Inn(Q)$ since $Q$ is radical. Then $P
\leq PQ \leq QC_S(Q) = Q$ with the equality because $Q$ is centric.
\end{proof}

The next two lemmas give applications of the Extension axiom.  The second is
useful for locating the $\F$-centric radicals that contain a given weakly
$\F$-closed subgroup.

\begin{Lem}\label{l:extfn} Let $P' \leq S$ be fully $\F$-normalized, and let
$P$ be a subgroup of $S$ which is $\F$-conjugate to $P'$. Then there exists a
morphism $\alpha \in \Hom_{\F}(N_S(P), N_{S}(P'))$ such that $P^\alpha = P'$.
\end{Lem} \begin{proof} See \cite[I.2.6(c)]{AschbacherKessarOliver2011}.
\end{proof}

\begin{Lem}\label{l:abovewc} Let $W$ be an $\F$-centric and weakly $\F$-closed
subgroup of $S$. For any subgroup $P$ of $S$ containing $W$, restriction
induces an isomorphism \[ \Aut_\F(P)/\Aut_W(P) \longrightarrow
N_{\Out_\F(W)}(\Out_P(W)) \] and therefore an isomorphism \[ \Out_\F(P)
\longrightarrow N_{\Out_\F(W)}(\Out_P(W))/\Out_P(W).  \] Hence, the map $P
\mapsto \Out_P(W)$ is a bijection between the collection of $\F$-centric
radical subgroups containing $W$ and the collection of subgroups of $\Out_S(W)$
which are $p$-radical in the group $\Out_\F(W)$.  \end{Lem} \begin{proof}
Consider the restriction map $\rho\colon \Aut_\F(P) \to
N_{\Aut_\F(W)}(\Aut_P(W))$, under which $\Aut_W(P)$ maps onto $\Inn(W)$ and
under which $\Inn(P)$ maps onto $\Aut_P(W)$.  Since $W$ is weakly closed, it is
fully $\F$-normalized by Lemma~\ref{l:extfn}. A direct application of the
extension axiom \cite[Defintion~1.2(II)]{BrotoLeviOliver2003} then gives that
$\rho$ is surjective.  Since $W$ is $\F$-centric, the centralizer in $\F$ of
the centric subgroup $W$ is the fusion system of $Z(W)$ by
Lemma~\ref{l:fusion-system-of-the-center}, so the kernel of $\rho$ is
$\Aut_{Z(W)}(P)$, which is contained in $\Aut_W(P) \subseteq \Inn(P)$. The
induced map \[ \Aut_\F(P)/\Aut_W(P) \longrightarrow
N_{\Aut_\F(W)}(\Aut_P(W))/\Aut_W(W) \cong N_{\Out_\F(W)}(\Out_P(W)) \] is an
isomorphism, and therefore upon factoring by $\Aut_P(P)/\Aut_W(P)$, the induced
map \begin{eqnarray} \label{e:outs-abovewc} \Out_\F(P) \longrightarrow
N_{\Aut_\F(W)}(\Aut_P(W))/\Aut_P(W) \cong N_{\Out_\F(W)}(\Out_P(W))/\Out_P(W)
\end{eqnarray} is an isomorphism. 

Observe that $W$ is normal in $S$ because it is weakly $\F$-closed. So
$\Out_P(W) \cong P/W$ since $C_S(W) \leq W$. The map $P \mapsto \Out_P(W)$ is
therefore a bijection between the subgroups containing $W$ and the subgroups of
$\Out_S(P)$. By \eqref{e:outs-abovewc}, $\Out_\F(P)$ corresponds to
$N_{\Out_\F(W)}(\Out_P(W))/\Out_P(W)$ under the bijection, so $P$ is
$\F$-radical if and only if $\Out_P(W)$ is $p$-radical in the group
$\Out_\F(W)$ (Remark~\ref{r:radical}). The last statement now follows because
the collection of $\F$-centric subgroups is closed under passing to overgroups.
\end{proof}

\subsection{Quaternion groups and the $2$-local structure of
$\SL_2(q)$}\label{s:quat} It will be convenient to recall here standard facts
about the $2$-local structure of $\SL_2(q)$, where $q$ is an odd prime power.
For reasons that will become apparent in a moment, we set $l \geq 0$ and take
$q = q_l = 5^{2^{l}}$ for simplicity of exposition.  Given this notation,
$\SL_2(q)$ has generalized quaternion Sylow $2$-subgroups of order $2^{l+3}$,
and this can be seen as follows.  First, the size of a Sylow $2$-subgroup can
be deduced from the order $q(q-1)(q+1)$ of $\SL_2(q)$, together with the fact
that the $2$-adic valuation $v_2(5^{2^l}-1)$ is $l+2$. 
By the choice of $q$, the multiplicative group $\mathbb{F}_q^\times$ contains a
primitive root of unity $\omega$ of order $2^{l+2}$. Thus, \[ a:=\left(
\begin{smallmatrix} \omega & 0  \\ 0 & \omega^{-1} \\ \end{smallmatrix} \right)
\quad \text{and} \quad  b:=\left( \begin{smallmatrix} 0 & -1  \\ 1 & 0  \\
\end{smallmatrix} \right) \] generate a Sylow $2$-subgroup of $\SL_2(q)$ by
order considerations. Since $a$ and $b$ satisfy the relations
\begin{equation}\label{eq:quat} a^{2^{l+2}}=b^4=1, \quad a^{2^{l+1}}=b^2, \quad
b^{-1}ab=a^{-1} \end{equation} we see that $R:=\langle a,b \rangle$ is a
generalized quaternion group of order $2^{l+3}$.  The following lemma records
some basic facts about the subgroup structure of a generalized quaternion
group. 

\begin{Lem}\label{l:quatprop} The following hold: 
\begin{itemize} 
\item[(a)] each element of $R$ is of the form $a^ib^j$ with $0 \le i \le
2^{l+2}-1$ and $0 \le j \le 1$; 
\item[(b)] each element in $R \backslash \langle a \rangle$ is of order 4;
\item[(c)] $a^ib$ is conjugate to $a^jb$ if and only if $i \equiv j \mod 2$,
where $0 \le i,j \le 2^{l+2}-1$; 
\item[(d)] the set $\Q$ of subgroups of $R$ isomorphic with $Q_8$ is given by
$\langle a^{2^l},a^ib \rangle$ with $0 \le i \le 2^{l+2}-1$; 
\item[(e)] when $l > 0$, there are two conjugacy classes of
$Q_8$-subgroups of size $2^{l-1}$ represented by $Q:=\langle a^{2^l},b \rangle$
and $Q':=\langle a^{2^l},ab \rangle$; and 
\item[(f)] when $l > 0$, $N_S(Q)=\langle Q,a^{2^{l-1}} \rangle$ and
$N_S(Q')=\langle Q',a^{2^{l-1}} \rangle$.  
\end{itemize} 
\end{Lem}
\begin{proof} 
Part (a) is clear, and (b) follows since, for each $i$, 
\[
(a^ib)^2=a^iba^ib=ba^{-i}a^ib=b^2 
\] 
has order 2. A general element $a^jb^m$ as in (a) conjugates $a^ib$ to 
\[
b^{-m}a^{-j}a^iba^jb^{m}=(b^{-m}a^{i-2j}b^{m})b=
\begin{cases} 
a^{i-2j}b, \mbox{ if $m=0$}\\ 
a^{2j-i}b, \mbox{ if $m=1$}.  
\end{cases} 
\] 
from which the claim in (c) follows. 

Let $\Q$ be the set of subgroups of $R$ isomorphic to $Q_8$ as in (d), and fix
$Q \in \Q$. As $\gen{a}$ is cyclic of index $2$ in $R$, we have $Q\gen{a} = R$,
and so $Q \cap \gen{a} = \gen{a^{2^l}}$ by order considerations. This shows
that $Q$ is of the form $\gen{a^{2^l},a^ib}$ for some $i$. Conversely, for each
$i$, the elements $a^{2^{l}}$ and $a^{i}b$ satisfy the relations
\eqref{eq:quat}, applied with $l=0$, in place of $a$ and $b$ respectively.
Hence, $\gen{a^{2^l},a^ib} \cong Q_8$, and so $\gen{a^{2^l}, a^ib} \in \Q$.
This completes the proof of (d). 

Note that exactly four elements of the form $a^ib$ lie in a given member of
$\Q$.  Since there are $2^{l+2}$ choices for $i$, $\Q$ has cardinality
$2^{l+2}/4 = 2^l$. Part (e) now follows from the conjugacy information in (c),
while (f) follows from the observation that $b^a=a^{-2}b$ so that
$b^{a^i}=a^{-2i}b$.  
\end{proof}

Since $v_2(q-1)=l+2$ and $\omega$ is a primitive $2^{l+2}$ root of unity,
$\sqrt{\omega} \notin \mathbb{F}_{q}^\times$. So $p(t):=t^2-\omega$ is
irreducible in $\mathbb{F}_q$, and $\mathbb{F}_0:=\mathbb{F}_q[t]/p(t)$ is a
finite field of order $q^2$ containing $\mathbb{F}_q$. Set $c:=\left(
\begin{smallmatrix} t & 0  \\ 0 & t^{-1}  \\ \end{smallmatrix} \right) \in
\SL_2(\mathbb{F}_0).$ Then straightforward computations show that $c^2=a$ (so
$c$ has order $2^{l+3}$) and that $c^{-1}bc=ba$ and $cbc^{-1}=ab$. Hence by
Lemma \ref{l:quatprop} (d),(e), $c$ fuses the two conjugacy classes of
subgroups of $R$ isomorphic with $Q_8$. 

Finally, we will need the following lemma, which we will usually use in
Section~\ref{s:solqcrel} without further comment. For a discussion of (2), see
for example \cite[Theorem~4.54]{CravenTheory}.
\begin{Lem}\label{l:normsl2q} 
Let $\F:=\F_R(\SL_2(q))$ be the fusion system of $\SL_2(q)$ with $q$ odd. 
\begin{enumerate}
\item If $l = 0$, then $\F$ is constrained with centric normal subgroup $R$,
$N_{\SL_2(q)}(R) \cong \SL_2(3)$, and $\Out_\F(R) \cong C_3$. 
\item If $l > 0$, then $\{R,Q,Q'\}$ is a complete set of $\F$-conjugacy class
representatives of $\F$-centric radical subgroups. Moreover, $N_{\SL_2(q)}(Q)
\cong N_{\SL_2(q)}(Q') \cong \GL_2(3)$, $\Out_\F(Q) \cong \Out_\F(Q') \cong
S_3$, and $\Out_{\F}(R) = 1$.  
\end{enumerate}
\end{Lem}

\subsection{$\Spin_7(q)$}\label{s:spin} 
Let $q$ be an odd prime power, and let $V$ be an odd dimensional vector space
over $\FF_q$. Let $\qq$ be a nondegenerate quadratic form on $V$ and $\bb$ the
associated symmetric bilinear form, which determine each other via $\qq(v) =
\bb(v,v)$ and $\bb(v,w) = \frac{1}{2}(\qq(v+w)-\qq(v)-\qq(w))$.  Let $(V,\qq)$
the associated geometric space, and $\mathrm{O}(V) = \mathrm{O}(V,\qq)$ the
isometry group.  There are two such forms $\bb$ up to equivalence, and the
corresponding isometry groups are isomorphic. We may therefore take $\bb$ to be
of square discriminant when convenient. We have $\mathrm{O}(V) = \{\pm 1\} \times
\SO(V)$. The spinor norm $\SO(V) \to \FF_q^{\times}/\FF_q^{\times 2}$ is
defined by writing an element of $\SO(V)$ as a product of reflections, and then
taking the product of the discriminants of the $-1$-eigenspaces of those
reflections.  The kernel of the spinor norm is the simple subgroup $\Omega(V)$.
Let $\Spin(V)$ be the perfect double cover of $\Omega(V)$, and write $Z$ for
the center of
$\Spin(V)$. 

We generally refer to \cite[Appendix~A]{LeviOliver2002} and
\cite[Section~4]{AschbacherChermak2010} for information on the construction and
subgroup structure of the $\Spin$ groups but record the following basic lemma
for use in Section~\ref{s:solqcrel}.

\begin{Lem} \label{l:invlift} 
An involution in $\Omega(V)$ lifts to an involution in $\Spin(V)$ if and only
if the dimension of its $-1$-eigenspace is a multiple of $4$.  
\end{Lem} 
\begin{proof} 
See \cite[Lemma~A.4(b)]{LeviOliver2002}.
\end{proof}

From now take $V$ to be of dimension $7$.  To help motivate some of the
definitions in the next subsection, we describe very roughly the structure of
the normalizer of a four subgroup containing $Z$ in $\Spin_7(q)$.  In this
case, Lemma~\ref{l:invlift} implies that $\Spin_7(q) := \Spin(V)$ has two
classes of involutions, namely those with representatives given by the central
involution $z \in Z(\Spin_7(q))$ and by the preimage of an involution with
$-1$-eigenspace of dimension $4$.  Let $V_1$ be a nondegenerate subspace of
dimension $4$ (and Witt index $2$) and let $V_2$ be its orthogonal complement.
Let $z_1 \in \Spin(V)$ be the involution with $-1$-eigenspace $V_1$ (an
involution by Lemma~\ref{l:invlift}). The normalizer $B :=
N_{\Spin(V)}(\gen{z,z_1})$ contains the normal subgroup $C_B(V_2)C_B(V_1)$ with
index $4$, isomorphic to the commuting product
\[
\Spin(V_1) * \Spin(V_2) \cong (\SL_2(q) \times \SL_2(q)\times \SL_2(q))/\gen{(-1,-1,-1)},
\]
There is a four group complementing $C_B(V_2)C_B(V_1)$ in $B$, which contains
an involution interchanging the first two $\SL_2(q)$'s and centralizing the
third (while acting as $-1$ on $V_2$), and which contains an involution acting
simultaneously as a diagonal automorphism on each $\SL_2(q)$ factor.

All additional information about $\Spin_7(q)$ that we require directly will be
collected later in Lemmas \ref{l:cse} through \ref{l:allcraboveT}, in
Proposition~\ref{p:ngpnleqk}, and in the proof of Lemma~\ref{l:QRandQR*}.

\subsection{Construction of $\Sol(q)$}\label{s:conssol} Following work of
Solomon \cite{Solomon1974}, the Benson-Solomon systems were predicted to exist
by Benson \cite{Benson1998c}, and then later constructed by Levi and Oliver
\cite{LeviOliver2002, LeviOliver2005}.  They are exotic in the sense that they
are not of the form $\F_S(G)$ for any finite group $G$ with Sylow $2$-subgroup
$S$. They are also not the fusion system of any $2$-block of a finite group
\cite{Kessar2006}, \cite[Section~9.4]{CravenTheory}, an a priori stronger
statement.  After Levi and Oliver, Aschbacher and Chermak gave a different
construction of the Benson-Solomon systems as the fusion system of a certain
free amalgamated product of two finite groups having Sylow 2-subgroup
isomorphic to $\Spin_7(q)$ \cite{AschbacherChermak2010}.  We primarily view
$\Sol(q)$ through the lens of \cite{AschbacherChermak2010}, so we consider it
as the $2$-fusion system of an amalgamated product $G=H *_B K$, where
$H:=\Spin_7(q)$. 

The isomorphism type of the Benson-Solomon system $\Sol(q)$ depends not on $q$,
but only (uniquely) on the $2$-adic valuation of $q^2-1$ by \cite[Theorem
3.4]{ChermakOliverShpectorov2008}. For reasons of exposition, it will be
helpful therefore to fix the following choice of $q$: unless otherwise
specified, for the remainder of this section and the next, we \[ \textit{let
$l$ be a fixed but arbitrary nonnegative integer, and set $q = 5^{2^{l}}$.} \]

We have described how $B$ arises as a subgroup of $H$ in
Subsection~\ref{s:spin} (but the explicit embedding $B \hookrightarrow H$ in
the amalgam is not the ``obvious'' one). We now take a more abstract approach
to obtain a working description of $K$ in Aschbacher-Chermak free amalgamated
product, as
follows. Consider the natural inclusion $\SL_2(q) \le \SL_2(q^2)$ induced by an
inclusion of fields, and define $N := N_{\SL_2(q^2)}(\SL_2(q))$ so that
$|N:\SL_2(q)|=2$ and $N$ and $\SL_2(q)$ both have generalized quaternion Sylow
$2$-subgroups, as explained more fully in Subsection~\ref{s:quat}.  Form the
wreath product $W:= N \wr S_3$, and let $N_0 := N_1 \times N_2 \times N_3$ and
$X := S_3$ be the base and acting group respectively.  Note that $O^2(N_0)
\norm W$ is a direct product $\hat{L}_1 \times \hat{L}_2 \times \hat{L}_3$ of
three copies of $\SL_2(q)$ permuted transitively by $X$. 

Define $\hat{K} := O^2(N_0)C_{N_0}(X)X$ regarded as the group generated by the
wreath product $O^2(N_0) \rtimes X$, and an element of $N_0 \backslash
O^2(N_0)$ acting in the same way simultaneously on each factor $\hat{L}_i$ of
$O^2(N_0)$.  Thus, $Z(O^2(N_0)) = Z(O^2(N_0)C_{N_0}(X)) = \gen{(\pm 1, \pm 1,
\pm 1)}$ and $Z(\hat{K}) = \gen{(-1, -1, -1)}$. Here, we write $1$ for the
identity matrix.  Finally, set \[ K := \hat{K}/Z(\hat{K}).  \] We will write
$[a_1,a_2,a_3]$, for example, for the image $K$ of an element $(a_1,a_2,a_3)$
of $O^2(N_0)C_{N_0}(X)$. 

\begin{Not}\label{N:Knotation1} We fix the following notation for certain
subgroups of $K$.  \begin{itemize} \item[(a)] $L_i \cong \SL_2(q)$ for $i = 1,
2, 3$ are the images in $K$ of the subgroups $\hat{L}_i$ of $\hat{K}$;
\item[(b)] $L_0 := L_1L_2L_3$; \item[(c)] $X \cong S_3$ is the image in $K$ of
the subgroup with the same name; \item[(d)] $\tau \in X$ is the permutation
$(1,2)$ on the indices of the $L_i$; \item[(e)] $S$ is a Sylow $2$-subgroup of
$K$ containing $\tau$; \item[(f)] $U = Z(L_0) = \gen{[\pm 1, \pm 1, \pm 1]}
\cong C_2 \times C_2$; and \item[(g)] $B := L_0S$.  \end{itemize} \end{Not}

Thus, the subgroup $B$ in Notation~\ref{N:Knotation1}(g) is a subgroup of $K$
of index $3$, and $B \cap X = \gen{\tau}$. As was shown in
\cite{AschbacherChermak2010} and recalled in the last subsection, there is a
four subgroup $U \leq H$ such that $B \cong N_{H}(U)$, and a choice of
injection $\iota:B \hookrightarrow H$ such that the free amalgamated product
$G=H *_B K$ has finite Sylow $2$-subgroup $S$ and determines a saturated fusion
system $\Sol(q)$ over $S$ that was constructed by Levi and Oliver by different
means \cite{LeviOliver2002,LeviOliver2005}. An incorrect choice of $\iota$ can
lead to a fusion system which is not saturated.  See \cite[Section
5]{AschbacherChermak2010} and \cite{LeviOliver2005} for more details, but
generally this subtlety will be unimportant in our computations.  

It will be helpful to introduce some more notation. Some of it follows the
notation of \cite[Section~10]{AschbacherChermak2010} in preparation for the
application in Section~\ref{s:solqcrel} of some of the results there.

\begin{Not}\label{N:Knotation2} 
We fix the following additional notation for subgroups and elements of $K$.
\begin{itemize} 
\item[(a)] $R_i \cong Q_{2^{l+3}}$ is a Sylow $2$-subgroup of $L_i$ for
$i=1,2,3$, chosen so that $X \cong S_3$ acts on the set $\{R_1,R_2,R_3\}$;
\item[(b)]  $R_0:=R_1R_2R_3 \in \Syl_2(L_0)$; 
\item[(c)]  $\Q_i$ is the set of subgroups of $R_i$ isomorphic to $Q_8$ for
$i=1,2,3$; thus, $\Q_i = \{R_i\}$ if $l = 0$, while $\Q_i$ is a union of two
$R_i$-conjugacy classes of subgroups if $l > 0$ by Lemma~\ref{l:quatprop}(e);
\item[(d)] when $l > 0$, $Q_i,Q_i' \in \Q_i$ are representatives for the
two $R_i$-conjugacy classes of subgroups chosen so that $X \cong S_3$ acts by
permuting the sets $\{Q_1,Q_2,Q_3\}$ and $\{Q_1',Q_2',Q_3'\}$;
\item[(e)] $\mathbf{c} := [c,c,c]$ where $c$ is as in Section~\ref{s:quat}, 
so that $\mathbf{c}$ acts simultaneously on $L_i \cong SL_2(q)$ by conjugation
in the way described there; 
\item[(f)] $\mathbf{d} := [b,b,b]\mathbf{c} \in K$, an involution commuting
with $\tau$; and 
\item[(g)] $\tau' = \mathbf{d}\tau$.  
\end{itemize} 
\end{Not}

Note that $\gen{\mathbf{d}, \tau}$ is a four group which intersects $R_0$
trivially. Thus, refining Notation~\ref{N:Knotation1}(e), we fix the following
Sylow $2$-subgroup of $K$ throughout the remainder of this section and in
Section~\ref{s:solqcrel}: \[ S = R_0\gen{\mathbf{d}, \tau}.  \] Then $R_0$ is
normal in $S$ with complement $\gen{\mathbf{d},\tau}$, $R_3$ is normal in $S$
and $\mathbf{d}$ interchanges the two $R_i$-conjugacy classes of subgroups
isomorphic with $Q_8$ when $l > 0$. Finally, we define 
\[ 
\K:=\F_S(K), \hspace{2mm} \H:=\F_S(H) \hspace{2mm} \mbox{ and }\hspace{2mm} \F:=\F_S(G).  
\]
We note that $\F$ is the fusion system generated by $\H$ and $\K$ by
\cite[Theorem~3.3]{Semeraro2014}, namely $\F$ is the smallest fusion system on
$S$ containing all morphisms in $\H$ and $\K$.

\subsection{The torus of $\F$}\label{s:torus} 
The next lemma calls attention to the 2-power torsion subgroup $T \leq S$ in a
maximal torus of $H$. Viewed as a subgroup of $K$, it may be generated by the
elements $[a,1,1], [1,a,1], [c,c,c]$ in the notation of
Subsection~\ref{s:quat}, and it is inverted by the involution $\mathbf{d}$.

\begin{Lem}\label{l:2torus} 
There is a unique subgroup $T$ of $S$ isomorphic to $(C_{2^{l+2}})^3$. The
centralizer $C_H(T)$ is a split maximal torus of $H$; in particular $C_S(T) =
T$. The subgroup $T$ is $\F$-centric and weakly $\F$-closed. Moreover,
$\Out_S(T) = S/T \cong C_2 \times D_8$, $\Out_\F(T) \cong C_2 \times \GL_3(2)$,
and $\Out_\H(T) \cong C_2 \times S_4$ is the maximal parabolic in $\Out_\F(T)$
lying over the Borel subgroup $\Out_S(T)$ given by the stabilizer of $Z$ in the
action of $\Out_\F(T)$ on $\Omega_1(T)$.
\end{Lem} 
\begin{proof} 
By \cite[Lemma 4.9(c)]{AschbacherChermak2010}, there is a unique homocyclic
subgroup of $S$ of rank $3$ and exponent $4$, $T$ is the centralizer in $S$ of
that subgroup, and $C_H(T)$ is a split maximal torus of $H$. Since $T$ is
abelian, this shows that $T$ is the unique subgroup of $S$ of its isomorphism
type. Then \cite[Lemma~4.3 and Lemma~4.8]{AschbacherChermak2010} show that $S/T
\cong C_2 \times D_8$, and $\Out_\H(T) \cong C_2 \times S_4$.  The structure of
the outer automorphism group $\Out_\F(T)$ follows from the construction of the
Aschbacher-Chermak amalgam in \cite[Lemma~5.2]{AschbacherChermak2010}.  All
other points follow.
\end{proof}

\subsection{The standard elementary abelian chain in $S$}\label{ss:standardseq} 
We refer to Sections~4 and 7 of \cite{AschbacherChermak2010} for more
discussion on the following items. Set $z := [-1,-1,1] = [1,1,-1] \in S$.
There is a chain of elementary abelian subgroups \[ Z < U < E < A \] of ranks
1, 2, 3, and 4, respectively, where $Z = Z(S) = \gen{z}$, $U$ is the unique
normal four subgroup of $S$ of Notation~\ref{N:Knotation1}(f), $E =
\Omega_1(T)=\langle [-1,1,1],[1,-1,1],[a^{2^l},a^{2^l},a^{2^l}]\rangle$, and $A
= E\gen{\mathbf{d}}$. For a member $X_n$ of the above chain of rank $n$,
$\Aut_\F(X_n) = \Out_\F(X_n) \cong \GL_n(2)$ by
\cite[Lemma~3.1]{LeviOliver2002}.  Also, $\H = C_\F(Z)$ and $\K = N_\F(U)$ by
\cite[Proposition~9.2]{AschbacherChermak2010}.

\subsection{Centric radicals containing the torus}\label{ss:abovetorus}
In the next five lemmas we identify, using Lemma~\ref{l:abovewc}, the outer
automorphism groups of the centric radical subgroups that contain the $2$-torus
$T$. 

\begin{Lem} \label{l:cse} 
The subgroup $C_S(E)$ of $S$ is $\F$-centric and weakly $\F$-closed, $C_S(E) =
T\gen{\mathbf{d}}$, and $C_H(C_S(E)) = Z(C_S(E))) = E$. Moreover,
$\Out_S(C_S(E)) = S/C_S(E) \cong D_8$, $\Out_\F(C_S(E)) \cong \GL_3(2)$, and
$\Out_{\H}(C_S(E)) \cong S_4$ is the maximal parabolic in $\Out_\F(C_S(E))$
given by the stabilizer of $Z$ under the natural action of $\Out_\F(C_S(E))$ on
$E$.  
\end{Lem} 
\begin{proof} 
Since $E = \Omega_1(T)$, we have $C_S(E) \geq T$. As $T$ is $\F$-centric,
so is $C_S(E)$. Let $\phi \in \Hom_\F(C_S(E),S)$.  By Lemma~\ref{l:2torus},
$T^\phi = T$, so also $E^\phi = E$. Hence, $C_S(E)^\phi \leq C_S(E^\phi) =
C_S(E)$, and so $C_S(E)$ is weakly $\F$-closed. 
From the description of $\Out_\F(T)$ in Lemma~\ref{l:2torus}, the kernel of the
action of $S/T$ on $E$ is of order $2$.  Now $\mathbf{d} \in S$ inverts $T$, so
centralizes $E = \Omega_1(T)$. Hence, $\mathbf{d}$ represents the lone
nontrivial coset of $C_H(T)$ in $C_H(E)$, whose elements invert the maximal
torus $C_H(T)$ of $H$ containing $T$ (see
\cite[Lemma~4.3(a,d)]{AschbacherChermak2010}).  So $C_S(E) =
T\gen{\mathbf{d}}$, and $C_H(C_S(E)) \leq T$ from Lemma~\ref{l:2torus}.  Hence,
the center $Z(C_S(E))$ is $C_H(C_S(E)) = C_{T}(\mathbf{d}) = E$.

As $O_2(\Out_\F(T)) = \Out_{C_S(E)}(T)$, the descriptions of the outer
automorphism groups in $\F$ and $\H$ follow from Lemmas~\ref{l:abovewc} and
\ref{l:2torus}.  
\end{proof}

\begin{Lem}\label{l:s} 
$N_H(S) = S$ and $\Out_\H(S) = \Out_\F(S) = \Out_\K(S) = 1$.
\end{Lem} 
\begin{proof} 
Since $C_S(E)$ contains its centralizer in $H$ from Lemma~\ref{l:cse}, so does
$S$. Then as the Sylow $2$-subgroups of $S_4$ and $\GL_3(2)$ are
self-normalizing, the lemma now follows from Lemmas~\ref{l:abovewc} and
\ref{l:cse}. 
\end{proof}

\begin{Lem} \label{l:csu} 
The subgroup $C_S(U)$ of $S$ is $\F$-centric and weakly $\F$-closed, and
$Z(C_S(U)) = U$.  The quotient $C_S(U)/C_S(E)$ is the unipotent radical of the
stabilizer in $\Out_\F(C_S(E))$ of $U$.  Thus, $\Out_S(C_S(U)) =
\Out_\H(C_S(U)) \cong C_2$ is induced by $\gen{\tau}$, and $\Out_\F(C_S(U))
\cong S_3$ is induced by $X$. 
\end{Lem} 
\begin{proof} 
From the structure of $\Out_\F(C_S(E))$ in Lemma~\ref{l:cse},
$\Out_{C_S(U)}(C_S(E)) = C_S(U)/C_S(E)$ is the unipotent radical of the
stabilizer of $U$ in the action of $\Out_\F(C_S(E))$ on $E$, so in particular
$Z(C_S(U)) = C_{E}(C_S(U)) = U$. The descriptions of the outer automorphism
groups now follow from Lemmas~\ref{l:abovewc} and \ref{l:cse} and the structure
of $\GL_3(2)$. 
\end{proof}

\begin{Lem}\label{l:csez} 
The subgroup $C_S(E/Z) = \{s \in S \mid [E,s] \leq Z\}$ is $\F$-centric and
weakly $\F$-closed, and $Z(C_S(E/Z)) = Z$. The quotient $C_S(E/Z)/C_S(E)$ is
the unipotent radical of the stabilizer in $\Out_\F(C_S(E))$ of $Z$ in the
natural action on $E$.  Thus, $\Out_{\H}(C_S(E/Z)) = \Out_\F(C_S(E/Z)) \cong
S_3$.  
\end{Lem} 
\begin{proof}
Observe that $C_S(E) \leq C_S(E/Z)$ and that $C_S(E/Z)/C_S(E)$ is the group of
transvections in $\Out_\F(C_S(E))$ on $E$ with center $Z$. So $C_S(E/Z)/C_S(E)$
is the unipotent radical of the stabilizer of $Z$. Also, as $Z(C_S(E)) = Z$
from Lemma~\ref{l:cse}, we have $Z(C_S(E/Z)) = C_{E}(C_S(E/Z)) = Z$.  Since
$C_S(E)$ is $\F$-centric, weakly $\F$-closed, and $\Aut_\H(C_S(E)) =
C_{\Aut_\F(C_S(E))}(Z)$, all points follow from Lemmas~\ref{l:abovewc} and
\ref{l:cse} as in the previous lemma.  
\end{proof}

\begin{Lem}\label{l:allcraboveT} 
The collection of $\F$-centric radical subgroups containing $T$ is $\{C_S(E),
C_S(U), C_S(E/Z), S\}$.  The collection of $\H$-centric radical subgroups
containing $T$ is $\{C_S(E/Z),S\}$.  
\end{Lem} 
\begin{proof} 
There are four $2$-radical subgroups in $\GL_3(2)$ inside a fixed Sylow
$2$-subgroup: the identity subgroup and the unipotent radicals of the three
associated parabolics.  So the lemma follows from the bijection of
Lemma~\ref{l:abovewc} together with Lemmas~\ref{l:cse}-\ref{l:csez}.  
\end{proof}

\subsection{The sectional rank of $S$} 
Before continuing, we record the sectional rank of $S$ using the later
Proposition~\ref{p:ngpnleqk}, which locates an extraspecial subgroup of order
$2^7$ in $S$. 
\begin{Lem} \label{L:sectrank} 
The sectional rank of $S$ is $6$.  
\end{Lem} 
\begin{proof} 
By Lemma~\ref{p:ngpnleqk}(a) below, $S$ contains an extraspecial subgroup with
central quotient of rank $6$, and hence $s(S) \geq 6$. On the other hand, the
sectional rank of a group is at most the sum of the sectional ranks of a normal
subgroup and corresponding quotient, so Lemma~\ref{l:2torus} shows that $s(S)
\leq s(T) + s(S/T) = 3 + 3 = 6$.  
\end{proof}

\section{Centric radicals in $\Sol(q)$}\label{s:solqcrel}

The aim of this section is to refine the description of the centric radical
subgroups of a Benson-Solomon system that results from a combination of
\cite[Section~10]{AschbacherChermak2010} and
\cite[Section~2]{ChermakOliverShpectorov2008}.  A starting point is the next
result due to Aschbacher and Chermak, which allows us to work in the groups $H$
and $K$ separately.  Adopt the notation from Section~\ref{s:solq}, and in
particular from Notation~\ref{N:Knotation1}, \ref{N:Knotation2}, and
Subsections~\ref{s:torus}, \ref{ss:standardseq}. Recall that $G$ is the
Aschbacher-Chermak free amalgamated product, and that $\F = \F_S(G)$.

\begin{Prop}\label{p:fcr}
Up to $\F$-conjugacy, a subgroup $P \leq S$ is $\F$-centric radical if and only if
\begin{itemize}
\item[(a)] $P = A$ is elementary abelian of order $2^4$ and $\Out_\F(P) = \GL_4(2)$; or
\item[(b)] $P=C_S(E)$ and $\Out_\F(P) \cong \GL_3(2)$; or
\item[(c)] Either:
\begin{itemize}
\item[(i)] $N_G(P) \le K$ and $P \in \K^{cr}$; or
\item[(ii)] $N_G(P) \le H$ and $P \in \H^{cr}$.
\end{itemize} 
\end{itemize}
\end{Prop}
\begin{proof}
See \cite[Lemma 10.9]{AschbacherChermak2010}.
\end{proof}
For the smallest Benson-Solomon system, the results of
\cite{ChermakOliverShpectorov2008}, when combined with Proposition~\ref{p:fcr},
supply sufficiently precise information for our needs, as we make clear in
Subsection~\ref{s:l=0}.  For the larger systems, Proposition~\ref{p:ngpnleqk}
below yields a sufficiently detailed description for the centric radicals
occurring in Proposition~\ref{p:fcr}(c)(ii) whose normalizer in $G$ is not
contained in $K$. 

Recall that $(V, \qq)$ is the orthogonal space from Section~\ref{s:spin};
$\qq(v)$ is referred to as the \emph{norm} of the vector $v$.  Following
\cite[Section~10]{AschbacherChermak2010}, we write $\Lambda(V)$ for the
collection of all sets of pairwise orthogonal subspaces whose sum is $V$.  For
$\Lambda \in \Lambda(V)$, the \emph{type} of $\Lambda$ is the nondecreasing
list of dimensions of the members of $\Lambda$. Write $N_H(\Lambda)$ for the
subgroup of $H$ which permutes the members of $\Lambda$, and write
$C_H(\Lambda)$ for the subgroup of $H$ which acts on each member of $\Lambda$.
We use exponential notation for the type, writing, for example, $1^7$ for
$(1,\dots,1)$ and $1^52$ for $(1,\dots,1,2)$.  Also we write $2^{1+2k}_{+}$ and
$2^{1+2k}_{-}$ for the extraspecial $2$-groups of width $k$ and plus and minus
type, respectively. Finally, if $Y$ is a finite group and $\pi$ is a set of
primes we write (as usual) $O_{\pi}(Y)$ for the unique maximal normal
$\pi$-subgroup of $Y$, $O_{\pi,\pi'}(Y)$ for the preimage in $Y$ of
$O_{\pi'}(Y/O_{\pi}(Y))$, and $O_{\pi,\pi',\pi}(Y)$ for the preimage in $Y$ of
$O_{\pi}(Y/O_{\pi,\pi'}(Y))$.

\begin{Prop}\label{p:ngpnleqk}
Suppose that $P \in \H^{cr}$ with $N_H(P) \nleq K$. Then, using $*$ to denote
the central product in which all centers have been identified, one of the
following holds.
\begin{itemize}
\item[(a)] $P = C_H(\Lambda)$ for some $\Lambda \in \Lambda(V)$ of type
$1^7$ with each member of $\Lambda$ spanned by a vector of square norm.
Moreover, $P \cong D_8*D_8*D_8 \cong 2^{1+6}_+$, and $\Out_\F(P) \cong
\Out_\H(P) \cong  
\begin{cases}
A_7 \mbox{ if $l=0$}\\
S_7 \mbox{ if $l > 0$};
\end{cases}$
\item[(b)] $P = C_H(\Lambda)$ for some $\Lambda \in \Lambda(V)$ of type
$1^7$ with exactly six $1$-spaces spanned by a vector of non-square norm. 
Moreover, $P \cong C_4*D_8*Q_8 = C_4*2^{1+4}_- \cong C_4 * 2^{1+4}_{+}$
and $\Out_\F(P) \cong \Out_\H(P) \cong S_6$;
\item[(c)] $l > 0$, and $P = O_2(N_H(\Lambda))\gen{t}$ for some $\Lambda
\in \Lambda(V)$ of type $1^52$ with each $1$-space spanned by a vector of
square norm and with the $2$-space a hyperbolic line. Moreover, $t$ acts as
$-1$ on the $1$-spaces and as a reflection on the line, $P \cong D_8*Q_8 *
Q_{2^{l+3}}$, and $\Out_\F(P) \cong \Out_\H(P) \cong S_5$; 
\item[(d)] $P = C_S(E/Z)$, $|S:P|=2$, and $\Out_\F(P) \cong \Out_\H(P) \cong
S_3$.
\end{itemize}
Moreover, there is exactly one $\H$-conjugacy class of subgroups of $S$ of each
of the given types.
\end{Prop}
\begin{proof}
Except for the last statement and the alternative descriptions of the groups
$P$ in (c) and (d), this is proved in \cite[Lemma 10.7]{AschbacherChermak2010}.
(Note each subgroup in (a)-(d) has center $Z$, so $\Out_\F(P) = \Out_\H(P)$
in all cases.)  To see that $P \cong D_8*Q_8*Q_{2^{l+3}}$ in (c), we recall
the setup of Aschbacher and Chermak as follows. 
Set $\bar{H} := H/Z \cong \Omega_7(q)$. 
The description of the subgroup in part (c) is discussed at and around
\cite[p.937, l.5]{AschbacherChermak2010}. For such a subgroup $P$ as in (c),
$P$ preserves a decomposition $V = V_1 \perp \cdots \perp V_5 \perp W$, where
$\dim V_i = 1$ and where $W$ is a hyperbolic line.  Set $V_0 = V_1+\cdots+V_5$,
$H_1 = C_{H}(W)$, $H_2 = C_H(V_0)$, and let $t$ be an element which acts as
$-1$ on the $V_i$ and which induces a reflection on $W$. The subgroup of $H$
preserving the above decomposition is of the form $H_1H_2\gen{t}$, where
$O_2(H_1) \cong D_8*Q_8$ has center $Z$, $H_1/O_2(H_1) \cong S_5$, and $H_2$ is
cyclic of order $2(q-1)$. 
Further, $P = P_1P_2\gen{t}$, where $P_1 = O_2(H_1)$, $P_2 = O_2(H_2)$, and
$[P_1,P_2\gen{t}] = 1$.
The image of $t$ in $\bar{H}$ has $-1$-eigenspace of dimension $6$, so $t$
squares to $z$ in $H$ by Lemma~\ref{l:invlift}. Likewise, an element $s \in
P_2$ acting as $-1$ on $W$ and as the identity on $V_0$ squares to $z$. This
shows $P_2\gen{t} = Q_{2^{l+3}}$ 
and $Z(P_1) = Z = Z(P_2\gen{t})$, so that $P$ has the structure as claimed in
(c).

The subgroup in (d) appears in the proof of 10.7 as the only subgroup $P$
satisfying the conditions that contains an elementary abelian normal subgroup
$P_0$ of rank at least 3. Having such $P_0$ of rank $\geq 4$ is ruled out on
\cite[p.957, l.19-22]{AschbacherChermak2010}. Let $P \in \H^{cr}$ with $N_H(P)
\nleq K$, and assume that there is an elementary abelian normal subgroup $P_0$
of $P$ of $2$-rank $3$.  Then \cite[p.973, l.22-29]{AschbacherChermak2010}
shows that $P_0 = E$ and $C_S(E) \leq P$. As $N_{H}(S) = S \leq K$ from
Lemma~\ref{l:s}, we have $P = C_S(E/Z)$ by Lemma~\ref{l:allcraboveT}.
Lemma~\ref{l:csez} then gives $\Out_\F(P) = \Out_\H(P) \cong S_3$. 

Finally, we must verify the last statement. For $P$ in (a)-(b), this follows
from a slight extension of Witt's Lemma, as stated in \cite[Lemma~2.7.2]{GLS3},
and induction on dimension. Consider a subgroup satisfying the conditions in
(c).  From the description of $P$ in the first paragraph, we see that $P$ is a
Sylow $2$-subgroup of $O_{2',2}(N_H(\Lambda))$. By \cite[Lemma~2.7.2]{GLS3}
again, $O_{2',2}(N_H(\Lambda))$ is uniquely determined up to $H$-conjugacy, so
$P$ is uniquely determined up to $H$-conjugacy by Sylow's Theorem in
$O_{2',2}(N_{H}(\Lambda))$. For uniqueness of the subgroup in (d), there is
nothing to do.  This completes the proof of the proposition.
\end{proof}

\begin{Not}\label{N:R**}
We denote a member of the $\F$-conjugacy class of a subgroup appearing in
Proposition~\ref{p:ngpnleqk} parts (a),(b), and (c) by $R_{1^7}$, $R_{1^7}'$
and $R_{1^52}$, respectively, to best indicate their origins.  The reader
should not confuse these with the generalized quaternion groups $R_1$,
$R_2$, and $R_3$.  When $l = 0$, the subgroups $R_{1^7}$ and $R_{1^7}'$
correspond with the subgroups $R$ and $R^*$ of Section 2 of
\cite{ChermakOliverShpectorov2008}.
\end{Not} 

We next describe the centric radical subgroups arising in case (c)(i) of
Proposition~\ref{p:fcr}. Recall Notation~\ref{N:Knotation1} and
Notation~\ref{N:Knotation2}. In addition, for any subgroup $Y$ of $K$, we set
$Y_0 = Y \cap L_0$, and let $Y_i$ be the \emph{projection} of $Y_0$ in $L_i$
for $1 \leq i \leq 3$. That is, $Y_i$ is the image in $L_i$ of the projection
of the preimage of $Y_0$ in $\hat{L}_i$ (cf. Notation~\ref{N:Knotation1}(a))
under the quotient map $\hat{K} \to K$. 

\begin{Prop}\label{p:10.2}
Fix $P \leq S$. Then $P \in \K^{cr}$ if and only if 
\begin{enumerate}
\item[(a)] $P \cap L_0 = P_1P_2P_3$, and for each $i \in \{1,2,3\}$, either
$P_i \in \Q_i$ or $P_i = R_i$; and
\item[(b)] one of the following holds. Either,
\begin{enumerate}
\item[(i)] $P \in \{C_S(U), S\}$, or
\item[(ii)] $P = P_1P_2P_3 \leq R_0$ with $P_i \in \Q_i$ for at least two indices $i$, or
\item[(iii)] $P = P_0\gen{s}$ for some $s \in P \backslash C_P(U)$ such that 
\begin{itemize}
\item[(1)] $s^2 \in P_0$,
\item[(2)] either $P_3 \in \Q_3$, or $P_i \in \Q_i$ for both $i = 1$ and $2$, and
\item[(3)] if $P_3 \in \Q_3$, then $\Out_{L_3}(P)$ is not a $2$-group.
\end{itemize}
\end{enumerate}
\end{enumerate}
\end{Prop}
\begin{proof}
This is part of \cite[Lemma~10.2]{AschbacherChermak2010}, namely (c) and (d) of
that lemma together with the statement beginning ``Conversely''.  The
requirement here in (b)(iii)(1) that $s$ square into $P_0$ does not appear in
\cite{AschbacherChermak2010}, but it is needed for the ``if'' part of the
proposition to hold in general. A patch for the proof of the ``if'' part in
\cite[Lemma~10.2]{AschbacherChermak2010} is given later in Remark \ref{r:102}.
\end{proof}

\subsection{The case $l=0$}\label{s:l=0}
An important distinguishing feature of the smallest Benson-Solomon system is
that $R_0$ is normal in the fusion system $\K$. When $l = 0$, this is most
naturally seen over $\mathbb{F}_3$, where a $Q_8$ Sylow $2$-subgroup is normal
in $\SL_2(3)$. Over $\mathbb{F}_5$, the normalizer of a quaternion Sylow
$2$-subgroup of $\SL_2(5)$ is $\SL_2(3)$, which still controls $2$-fusion in
$\SL_2(5)$ (c.f. Lemma~\ref{l:normsl2q}).  It will therefore be convenient to
treat the cases $l=0$ and $l > 0$ separately.  So assume here that $l = 0$. We
adopt the previous notation, except that we set 
\begin{equation}
\label{e:Qforl=0}
Q := R_0=R_1R_2R_3=Q_1Q_2Q_3
\end{equation}
in this smallest case so that $Q$, $R_{1^7}$ and $R_{1^7}'$ correspond with the
groups ``$Q$'', ``$R$'' and ``$R^*$'' considered in \cite[Section
2]{ChermakOliverShpectorov2008}.   

The next proposition lists the $\K$-centric radicals when $l = 0$, and does not
require Proposition~\ref{p:10.2}. 

\begin{Prop}\label{p:kl=0}
Let $l = 0$ and $P \in \K^{cr}$. Then exactly one of the following holds.
\begin{enumerate}
\item[(a)] $P = S$, and $\Out_\K(P) = 1$; 
\item[(b)] $P = Q$, and $\Out_\K(P) \cong (C_3)^3 \overset{-1 \times \wr}{\rtimes} (C_2 \times S_3)$; 
\item[(c)] $P 
= Q\gen{\tau}$, and $\Out_\K(P) \cong (C_3 \times C_3) \overset{-1}{\rtimes} C_2$;
\item[(d)] $P 
= Q\gen{\tau'}$, and $\Out_\K(P) \cong S_3$; or 
\item[(e)] $P = C_S(U) = Q\gen{\mathbf{d}}$, and $\Out_\K(P) \cong S_3$.
\end{enumerate}
\end{Prop}
\begin{proof}
As $Q$ is a centric normal $2$-subgroup of $\K$, it is contained in
every member of $\K^{cr}$ by Lemma~\ref{l:normcentrad}.  Now $S/Q$ is
a four group (the four group $\gen{\mathbf{d}, \tau}$ is a complement to
$Q$ in $S$), so there are only five possible centric radical
subgroups.  Since $O^2(K) \cap S = Q$, if two distinct subgroups of
$S$ containing $Q$ were $\K$-conjugate, then two distinct subgroups of
the abelian group $S/Q$ would be $K/O^2(K) \cong
S/Q$-conjugate.  Since this is not the case, no two distinct subgroups
of $S$ containing $Q$ are $\K$-conjugate.  Next, from the definition
of $U$, both $Q$ and $\mathbf{d}$ centralize $U$ while $\tau$ does
not, so we must have $Q\gen{\mathbf{d}}=C_S(U)$. This shows the
equality in (e).

The structure of the outer automorphism groups are computable from knowledge of
$\Out_\K(Q)$: note that from the structure of $K$ (cf.
Lemma~\ref{l:normsl2q}), 
\[
\Out_{\K}(Q) \cong (C_3)^3 \rtimes (C_2 \times S_3)
\]
is a split extension of the wreath product $C_3 \wr S_3$ by the group generated
by the class $\left[c_\mathbf{d}\right] \in \Out_\K(Q)$ of conjugation
by $\mathbf{d}$ acting by inversion on the base. As $Q$ is weakly
$\K$-closed and centric, $\Out_\K(P) \cong
N_{\Out_\K(Q)}(\Out_P(Q))/\Out_P(Q)$ for each
overgroup $P$ of $Q$ in $S$ by Lemma \ref{l:abovewc}. From a
computation in the group $(C_3)^3 \overset{-1 \times \wr}{\rtimes} (C_2 \times
S_3)$ one sees for example that
\[
N_{\Out_\K(Q)}(\langle \left[ c_\tau \right] \rangle ) =
C_{\Out_\K(Q)}(\langle \left[ c_\tau \right] \rangle ) \cong 
(C_3)^2 \overset{-1, 1}{\rtimes} (C_2 \times C_2), 
\] 
where the acting group is given by $\langle \left[c_{\mathbf{d}} \right]\rangle
\times \langle \left[c_{\tau} \right]\rangle$. This shows that
$\Out_\K(Q\gen{\tau}) = N_{\Out_\K(Q)}(\langle \left[ c_\tau
\right] \rangle )/\gen{\left[ c_\tau \right]} \cong (C_3 \times C_3)
\overset{-1}{\rtimes} C_2$ as claimed. Cases (d) and (e) are handled similarly.
Visibly no resulting outer automorphism group has a nontrivial normal
$2$-subgroup, so all the candidate subgroups are $\K$-centric radical.
\end{proof}

\begin{Prop}\label{p:sol3}
Let $l = 0$.  Then, up to conjugacy, the $\K$-, $\H$-, and $\F$-centric radical
subgroups of $S$ together with their orders and automorphism groups appear in
Table \ref{t:1}, where a `$-$' indicates that the subgroup is not
centric radical in that fusion system.

\begin{table}[h]
\renewcommand{\arraystretch}{1.5}
\centering
\caption{$\Sol(5)$-conjugacy classes of $\Sol(5)$-centric radical subgroups}
\label{t:1}
\mbox{}\newline
\begin{tabular}{|c|c|c|c|c|}
\hline
$P$                & $|P|$  & $\Out_\H(P)$           & $\Out_\K(P)$    & $\Out_\F(P)$     \\ \hline
$S$                & $2^{10}$ & $1$                    & $1$             & $1$     \\ \hline
$Q$                & $2^8$  & $(C_3)^3 \rtimes (C_2 \times C_2)$ & $(C_3)^3 \rtimes (C_2 \times S_3)$ & $(C_3)^3 \rtimes (C_2 \times S_3)$ \\ \hline
$Q\gen{\tau}$ & $2^9$  & $(C_3 \times C_3) \overset{-1}{\rtimes} C_2$                & $(C_3 \times C_3) \overset{-1}{\rtimes} C_2$         & $(C_3 \times C_3) \overset{-1}{\rtimes} C_2$    \\ \hline
$Q\gen{\tau'}$ & $2^9$  & $S_3$              & $S_3$       & $S_3$       \\ \hline
$C_S(U)$              & $2^9$  & $-$                    & $S_3$       & $S_3$   \\ \hline
$R_{1^7}$                & $2^7$  & $A_7$              & $-$             & $A_7$      \\ \hline
$R_{1^7}'$              & $2^6$  & $S_6$              & $-$             & $S_6$       \\ \hline
$C_S(E/Z)$ & $2^9$  & $S_3$              & $-$             & $S_3$      \\ \hline
$C_S(E)$ & $2^7$  & $-$                    & $-$             & $\GL_3(2)$  \\ \hline
$A$           & $2^4$  & $-$                    & $-$             & $\GL_4(2)$    \\ \hline
\end{tabular}
\end{table}

\end{Prop}   
\begin{proof}
By Proposition~\ref{p:kl=0}, the column for $\K$ is correct. By
\cite[Lemma~2.1]{ChermakOliverShpectorov2008} and
\cite[Lemma~A.11(e,f)]{LeviOliver2002}, the column for $\H$ is correct.  We work
up to $\F$-conjugacy in what follows. Let $P \in \F^{cr}$.  By
Proposition~\ref{p:fcr}, either $P$ is listed in the last two rows of
Table~\ref{t:1}, or one of the following holds: (1) $P \in \K^{cr}$ and
$\Out_\F(P) = \Out_\K(P)$, or (2) $N_G(P) \nleq K$, $P \in \H^{cr}$ and
$\Out_\F(P) = \Out_{\H}(P)$.  If (1) holds, then $P$ is listed in the
first five rows of the table by Proposition~\ref{p:kl=0}.  If
(1) does not hold, then (2) holds, $P$ is listed in the next three
rows of the table, where the entries follow from Proposition
\ref{p:ngpnleqk}(a,b,d).

That no additional $\F$-conjugacy can occur between these subgroups can be seen
in several ways, one of which as follows.  Only three subgroups have pairwise
equal orders and isomorphic outer automorphism groups in $\F$, namely
$Q\gen{\tau'}$, $C_S(U)$, and $C_S(E/Z)$.

By Lemma~\ref{l:cse}, $C_S(E/Z)$ has center $Z$. 
Likewise, since $Z(Q) = U$, we have $Z(Q\gen{\tau'}) =
C_{U}(\tau') = Z$.  So as $U \leq Z(C_S(U))$, it follows that $C_S(U)$ is not
$\F$-conjugate to either of the other two subgroups.

Finally, note that $C_S(E/Z)$ contains the torus $T$. On the other hand, from
the description of $T$ in Section~\ref{s:torus}, we see that $Q \cap T
= \gen{[a,1,1],[1,a,1],[c^2,c^2,c^2]}$ is of index $2$ in $T$. As each element
in the coset $Q\tau'$ is nontrivial on $Q \cap T$ and $T$ is
abelian, it follows that $Q\gen{\tau'} \cap T$ is still of index $2$
in $T$.  So $C_S(E/Z)$ contains $T$, but $Q\gen{\tau'}$ does not.
Since $T$ is weakly $\F$-closed (Lemma~\ref{l:2torus}), the subgroups
$C_S(E/Z)$ and $Q\gen{\tau'}$ are not $\F$-conjugate.
\end{proof} 

We end this subsection with two lemmas in the case $l=0$ which will be needed
later. 

\begin{Lem}\label{l:qrweakly}
Each member of $\F^{cr}-\{A\}$ is weakly $\F$-closed when $l=0$.
\end{Lem}
\begin{proof}
The subgroup $S$ is clearly weakly closed, and $C_S(E)$, $C_S(U)$, and
$C_S(E/Z)$ were shown to be weakly $\F$-closed in Lemmas~\ref{l:cse},
\ref{l:csu}, and \ref{l:csez}. Let $P$ be one of the remaining subgroups, but
not $A$. By Proposition~\ref{p:sol3}, $P$ is centric and radical in $\H$, and
either $P = Q$ or $Z(P) = Z$.  The quotient $P/Z$ is centric and radical in
$\H/Z$ by \cite[Lemma~A.11(e)]{LeviOliver2002}.  Hence, $P/Z$ is weakly
$\H/Z$-closed by \cite[Lemma~2.1]{ChermakOliverShpectorov2008}.  It follows
that $P$ is weakly $\H$-closed.  Since $Q$ is normal in $\K$, it is weakly
$\K$-closed. Hence, $Q$ is weakly $\F$-closed since $\H$ and $\K$ are fusion
systems over $S$ which generate $\F$ (end of Section~\ref{s:conssol}). 

We are reduced to the case in which $Z(P) = Z$. Assume on the contrary that $P$
is not weakly $\F$-closed. By Alperin's Fusion Theorem
\cite[Theorem~A.10]{BrotoLeviOliver2003}, there is an overgroup $Y \leq S$ of
$P$ and an automorphism $\alpha \in \Aut_\F(Y)$ such that $P^\alpha \neq P$.
Then $Z(Y) \leq Z(P) = Z$, as $P$ is centric, so that $Z(Y) = Z$ is centralized
by $\alpha$. That is, $\alpha \in \H$. But then $P^\alpha = P$ by the previous
paragraph, a contradiction. 
\end{proof}

\begin{Lem}\label{l:QRandQR*}
$QR_{1^7} = Q\gen{\tau}$ and $QR_{1^7}' =
Q\gen{\tau'}$ when $l = 0$. 
\end{Lem}
\begin{proof}
This is a statement depending on $H$ only. Since $l=0$, $q = 5$. Write
$\bar{H}$ for $H/Z$. Fix a decomposition
\[
V = \ell_1 \perp \ell_2 \perp \ell_3 \perp \gen{x_7}
\]
with the following properties (\cite[cf. 4.4,4.6]{AschbacherChermak2010}):
\begin{enumerate}
\item each $\ell_i = \gen{x_{2i-1}, x_{2i}}$ is a hyperbolic line (i.e.,
$\qq(x_{2i-1}) = 0 = \qq(x_{2i})$, $\bb(x_{2i-1}, x_{2i}) = 1$), and
$\qq(x_7) = 1$. 
\item $\ell_1 \perp \ell_2 = \gen{x_1,x_4} \oplus \gen{x_3,x_2}$, with each
summand on the right side a natural $\FF_5L_1$-module; in particular, 
$[a,1,1]$ and $[b,1,1]$ act via the matrices 
\begin{align*}
[a,1,1] &\mapsto \smat{2&0&0&0\\0&-2&0&0\\0&0&2&0\\0&0&0&-2} & [b,1,1] &\mapsto \smat{0&0&0&-1\\0&0&1&0\\0&-1&0&0\\1&0&0&0}
\end{align*}
with respect to the basis $\{x_1,x_2,x_3,x_4\}$. 
\item $\ell_1 \perp \ell_2 = \gen{x_1,x_3} \oplus \gen{x_4,x_2}$, with each
summand on the right side a natural $\FF_5L_2$-module; in particular, $[1,a,1]$
and $[1,b,1]$ act via the matrices 
\begin{align*}
[1,a,1] &\mapsto \smat{2&0&0&0\\0&-2&0&0\\0&0&-2&0\\0&0&0&2} & [1,b,1] &\mapsto \smat{0&0&-1&0\\0&0&0&1\\1&0&0&0\\0&-1&0&0}
\end{align*}
with respect to the basis $\{x_1,x_2,x_3,x_4\}$. 
\item $\ell_3 \perp \gen{x_7}$ is the $3$-dimensional orthogonal module for
$L_3 \cong \Spin_3(5)$. We may view it as the module in which $L_3$ acts by
conjugation on $2\times 2$ trace zero matrices $M_2^0(\FF_5)$ with quadratic
form given by the determinant, via the isometry $M_2^0(\FF_5) \longrightarrow
\ell_3 \perp \{x_7\}$ defined by
\[
\left\{\smat{0&0\\2&0}, \smat{0&-1\\0&0}, \smat{2&0\\0&-2}\right\} \longmapsto \{x_5,x_6,x_7\}.
\]
Under this identification, 
$[1,1,a]$ acts via the matrix $\diag(-1, -1, 1)$ with respect to the ordered
basis $\{x_5,x_6,x_7\}$, and 
$[1,1,b]$ acts via the matrix $\smat{0&2&0\\-2&0&0\\0&0&-1}$. 
\end{enumerate} 

Next, define $u_{j}$ and $v_{j}$ via
\begin{align*}
u_{2i-1} &= x_{2i-1}-2x_{2i}, &  v_{2i-1} &= u_{2i-1}+u_{2i},\\
u_{2i}   &= -2x_{2i-1}+x_{2i}, &  v_{2i} &= u_{2i-1}-u_{2i},\\
u_7      &= x_7,              &       v_7 &= u_7.
\end{align*}
Thus, $\{u_1,\dots,u_7\}$ is an orthonormal basis for $V$, and
$\{v_1,\dots,v_7\}$ is an orthogonal basis such that $\qq(v_i) = 2 \notin
\FF_5^{\times 2}$ for each $i = 1,\dots,6$. The decompositions
\[
\Lambda = \{\gen{u_i} \mid i \in \{1,\dots,7\}\}, \qquad \text{ and } \qquad \Lambda' = \{\gen{v_i} \mid i \in \{1,\dots,7\}\}
\]
of $V$ are therefore of the type appearing in Proposition~\ref{p:ngpnleqk}(a)
and (b), respectively. The centralizers of the decompositions are 
\begin{align*}
C_H(\Lambda) &= \{e \in \Spin_7(5) \mid (u_i)\bar{e} = \pm u_i, i = 1,\dots,7\}, \text{ and}\\
C_{H}(\Lambda') &= \{f \in \Spin_7(5) \mid (v_i)\bar{f} = \pm v_i, i = 1,\dots,7\}.
\end{align*}

Observe from the definition of the $v_i$ that $C_H(\Lambda) \leq
N_H(\Lambda')$.  Similarly, it is a straightforward computation to see using
(2)-(4) that $Q$ acts on the sets $\Lambda$ and $\Lambda'$, i.e.
$Q \leq N_H(\Lambda)$ and $Q \leq N_H(\Lambda')$. It follows
that $QC_H(\Lambda)C_H(\Lambda)'$ is a $2$-subgroup of $H$.  Hence, we
may choose $h \in H$ with $(QC_H(\Lambda)C_H(\Lambda'))^h \leq S$. But
$Q \leq S$ and so $Q^h = Q$ by
Lemma~\ref{l:qrweakly}.  Likewise, it follows from Lemma~\ref{l:qrweakly} that
$C_H(\Lambda)^h = R_{1^7}$ and $C_H(\Lambda')^h = R_{1^7}'$. Replacing
$S$ with $S^{h^{-1}}$ if necessary, we may assume that $R_{1^7} =
C_H(\Lambda)$ and $R_{1^7}' = C_H(\Lambda')$. 

For a subset $I \subseteq \{1,\dots,7\}$, write $e_I$ for a fixed element of
$C_H(\Lambda)$ which maps $u_i \mapsto -u_i$ if $i \in I$, and which fixes
$u_i$ otherwise. When $I \subseteq \{1,\dots,6\}$, denote by $f_I$ an analogous
element of $C_H(\Lambda')$ with respect to the $v_i$'s.  A computation of the
action of $Q$ with respect to the bases $\{u_i \mid 1 \leq i \leq 7\}$
and $\{v_i \mid 1 \leq i \leq 7\}$ using (2)-(4) yields
\begin{align*}
Q \cap R_{1^7} &= \gen{[-1,1,1], [b,ab,1], [ab,b,1], [1,1,a], [1,1,b]}\\ 
                           &= \gen{e_{1234},  e_{13},  e_{14},  e_{56},  e_{57}}\\
                           &\cong C_2 \times (Q_8*Q_8),
\end{align*}
and
\begin{align*}
Q \cap R_{1^7}' &= \gen{[-1,1,1], [b,b,1], [ab,ab,1], [1,1,a]}\\
                             &= \gen{f_{1234}, f_{23},  f_{13},  f_{56}}\\
                             &\cong C_2 \times (Q_8*C_4),
\end{align*}
where here we have used Lemma~\ref{l:invlift} and the identity $[e,f] = (ef)^2$
to determine the isomorphism types.  The order $|Q \cap R_{1^7}| =
2^6$, and so $|QR_{1^7}| =
\frac{|Q||R_{1^7}|}{|Q \cap R_{1^7}|}= 2^9$.
Similarly, $|Q \cap R_{1^7}'| = 2^5$, so also
$|QR_{1^7}'| = 2^9$. 

We have shown that $\{QR_{1^7}, QR_{1^7}'\} \subset
\{Q\gen{\mathbf{d}}, Q\gen{\tau}, Q\gen{\tau'}\}$.
The involution $e_{4567} \in R_{1^7}-Q$ (Lemma~\ref{l:invlift}) acts as
$-1$ on $\ell_3 \perp \gen{x_7}$, so centralizes $L_3$. It also interchanges
the one-dimensional subspaces $\gen{x_3}$ and $\gen{x_4}$ while centralizing
the line $\ell_1$, and hence from (2)-(3) it interchanges $L_1$ and $L_2$ by
conjugation. It follows that $QR_{1^7} = Q\gen{\tau}$,
since neither $Q\gen{\mathbf{d}}$ nor $Q\gen{\tau'}$ have
such an element.  

Finally, we show that $QR_{1^7}' = Q\gen{\tau'}$. First, since $f_{1234} \in
U-Z$ does not commute with $f_{45}$ by Lemma~\ref{l:invlift}, it follows that
$QR_{1^7}'$ is not contained in $C_S(U) = Q\gen{\mathbf{d}}$.  Next, observe
that in contrast to the previous case, $C_{R_{1^7}'}(L_3) =
C_{R_{1^7}'}(\ell_3+\gen{x_7})$ (for example, note that ``$f_{4567}$'' has
nontrivial spinor norm). The group $C_{R_{1^7}'}(L_3) = \gen{f_{12}, f_{13},
f_{14}}$ induces the permutation group $\gen{(1,2)(3,4)}$ on
$\{\gen{x_1},\gen{x_2},\gen{x_3},\gen{x_4}\}$, and hence
$C_{R_{1^7}'}(L_3)$ acts on $L_1$ and $L_2$ by (2)-(3).  Therefore,
$QR_{1^7}'$ has no element centralizing $L_3$ and interchanging
$L_1$ and $L_2$, and so $QR_{1^7}' = Q\gen{\tau'}$.
\end{proof}

\subsection{The case $l > 0$}\label{s:kcr}
In this subsection, we determine a set of representatives for the
$\D$-conjugacy classes of elements in  $\D^{cr}$ for $\D \in \{\K,\H,\F\}$, in
the case when $l > 0$. First, we treat the case $\D=\K$.

\begin{Prop}\label{p:k}
Suppose that $l > 0$. There are eleven $\K$-conjugacy classes of elements of
$\K^{cr}$. Representatives of these classes together with their outer
automorphism groups in $\K$ are listed in Table \ref{t:4}.
\end{Prop}

\begin{table}[h]
\centering
\renewcommand{\arraystretch}{1.5}
\caption{$\K$-conjugacy classes of $\K$-centric radical subgroups, $l > 0$}
\label{t:4}
\begin{tabular}{|c|c|c|}
\hline
$P$                             & $|P|$    & $\Out_\K(P)$                                              \\ \hline
$S$                             & $2^{10+3l}$ & $1$                                                   \\ \hline
$C_{S}(U)$                          & $2^{9+3l}$ & $S_3$                                     \\ \hline
$Q_1Q_2Q_3$                     & $2^8$    & $S_3 \wr S_3$                   \\ \hline
$Q_1Q_2Q_3'$                    & $2^8$    & $(S_3 \wr C_2) \times S_3$ \\ \hline
$Q_1Q_2R_3$                     & $2^{8+l}$    & $S_3 \wr C_2$                    \\ \hline
$Q_1Q_2'R_3$                    & $2^{8+l}$    & $S_3 \wr C_2$                 \\ \hline
$Q_1Q_2Q_3\langle\tau\rangle$   & $2^9$    & $S_3 \times S_3$                        \\ \hline
$Q_1Q_2Q_3'\langle\tau\rangle$  & $2^9$    & $S_3 \times S_3$                        \\ \hline
$Q_1Q_2R_3\langle\tau\rangle$   & $2^{9+l}$ & $S_3$                                      \\ \hline
$Q_1Q_2'R_3\langle\tau'\rangle$ & $2^{9+l}$ & $S_3$                           \\ \hline
$R_1R_2Q_3\langle\tau\rangle$   & $2^{9+2l}$ & $S_3$                                       \\ \hline
\end{tabular}
\end{table}

\begin{proof}
Let $P \leq S$ be a centric radical subgroup of $\K$, taken up to
$\K$-conjugacy. We proceed through the possibilities in the description of
$\K^{cr}$ given by Proposition~\ref{p:10.2} and refer to the labelings of the
three cases given there.  If $P$ occurs in (b)(i), then $P$ is listed in the
first two rows of the table. By Lemma~\ref{l:s}, $\Aut_\K(S)
= \Inn(S)$ so that $\Out_\K(S) = 1$. Also, $C_S(U) = R_0\gen{\mathbf{d}}$, so
that $\Out_\K(C_S(U)) \cong S_3$ is induced by $X$. 

Consider a subgroup $P$ in (b)(ii). First assume that $P_i \in \Q_i$ for all
$i$.  Upon conjugating in $L_0$, we may assume that $P_i = Q_i$ or $Q_i'$ for
each $i$. Conjugating by $\mathbf{d}$, which interchanges $Q_i$ and $Q_i'$ for
each $i$, we may assume that there is at most one $Q_i'$ among the $P_i$'s.
Finally, we may conjugate by elements of $X$ to see that $P$ is one of the
subgroups in rows 3 and 4 of the table.

To compute $\Out_{\K}(P)$, observe that if $t \in L_0X$, then $P^t$ and $P$
have the same number of components $P_i^t$ which are $L_i$-conjugate to $Q_i$,
while $P^\mathbf{d}$ has three minus the number for $P$. This shows that
$N_{K}(P) = N_{L_{0}X}(P)$.  Thus, if $P = Q_1Q_2Q_3$, then $N_K(P) =
(N_{L_1}(Q_1)N_{L_2}(Q_2)N_{L_3}(Q_3))X$ and we see that $\Out_\K(P) \cong S_3
\wr S_3$ by Lemma~\ref{l:normsl2q}. Likewise if $P = Q_1Q_2Q_3'$, then $N_K(P)
= N_{L_1}(Q_1)N_{L_2}(Q_2)N_{L_3}(Q_3')\gen{\tau}$, so that $\Out_\K(P) \cong
(S_3 \wr C_2) \times S_3$.

Next, assume that $P_i \in \Q_i$ for exactly two indices $i$. Then as before,
we may conjugate so that $P = Q_1Q_2R_3$ or $Q_1Q_2'R_3$ is on the table.
Appealing to Lemma~\ref{l:normsl2q} again to see that $\Out_{L_3}(R_3) = 1$, we
have in the former case that $\Out_{\K}(P) \cong S_3 \wr C_2$ with the class of
$\tau$ wreathing, while in the latter case we have a similar situation with the
class of $\tau'$ wreathing. This concludes the case (b)(ii). 

Consider now a subgroup $P$ in (b)(iii), and recall that $Z(L_0) = U$.  Thus $P
= P_0\gen{s}$ with $s \in P-C_P(U)$ normalizing $P_0$.  Set $N = N_K(P)$ and $M
= N_K(P_0)$. Denote quotients modulo $P_0$ with bars. We set $M^+ =
\bar{M}/O_{2'}(\bar{M})$ and write quotients modulo $O_{2'}(\bar{M})$ with pluses.
Thus, for any subgroup $Y \leq M$, we write $Y^+$ for the image of $Y$ modulo
the preimage of $O_{3}(\bar{M})$ in $M$.

Since $L_0 \norm K$, we see that $P_0 = P \cap L_0 \norm N$ so that $N \leq M$.
In particular, $\bar{N}$ is defined. Also, since $\bar{s}$ is of order $2$, $N$
is the preimage in $M$ of $C_{\bar{M}}(\bar{s})$. As $P$ is radical, we must
have
\begin{eqnarray}
\label{e:Prad}
\gen{\bar{s}} = O_2(C_{\bar{M}}(\bar{s})).
\end{eqnarray}

We consider separately the cases where $P_0 \notin \K^{cr}$ and where $P_0 \in
\K^{cr}$. Assume first that $P_0 \notin \K^{cr}$, the easier case. Upon
comparing the conditions in (b)(ii) and (b)(iii), we have by our assumption
that $P_3 \in \Q_3$ and $P_i = R_i$ for $i = 1, 2$. Thus, $\bar{M} =
\gen{\bar{\tau}} \times \bar{N_{L_3}(P_3)} \cong C_2 \times S_3$, and so $P =
P_0\gen{\tau}$ by \eqref{e:Prad}. Hence, $\Out_\F(P) \cong S_3$, and $P$
appears in the last row of the table.

Assume next that $P_0 \in \K^{cr}$, so that $P_0$ is conjugate to a subgroup
considered in (b)(ii), rows 3-6 of the table. First assume that $P_0$
itself appears in rows 3-6. Our description of the normalizer in $K$ of $P_0$
in a previous paragraph together with order considerations imply that
$N_{R_0}(P_0)\gen{\tau}/P_0$ is a Sylow $2$-subgroup of $\Out_{\K}(P_0)$ if
$P_0$ appears in rows 3-5, and that $N_{R_{0}}(P_0)\gen{\tau'}/P_0$ is a Sylow
$2$-subgroup of $\Out_\K(P_0)$ if $P_0$ appears in row 6.  This shows that
$\Aut_{S}(P_0)$ is a Sylow $2$-subgroup of $\Aut_{\K}(P_0)$, that is, $P_{0}$
is fully $\K$-automized \cite[I.2.2]{AschbacherKessarOliver2011}.  Since $P_0$
is $\K$-centric, it is fully $\K$-centralized
(\cite[I.3.1]{AschbacherKessarOliver2011}). Hence, $P_0$ is fully
$\K$-normalized by \cite[I.2.6(c)]{AschbacherKessarOliver2011}.

Thus, by Lemma~\ref{l:extfn} (and since $P_0 \norm P$),  we may in any
case replace $P$ by a $\K$-conjugate and assume that $P_0$ is in rows 3-6 of
the table.
\newline \newline
\textit{Case 1.} Assume $P_0 = Q_1Q_2Q_3$ and recall Lemma~\ref{l:quatprop}(e).

Here, $\bar{M} = \bar{N_{L_0}(P_0)}\bar{X} \cong S_3 \wr S_3$,
$\bar{N_{R_0}(P_0)}^+ = O_{2}(M^+)$, and $\bar{N_{S}(P_0)} =
\bar{N_{R_0}(P_0)}\gen{\bar{\tau}}$. 
By assumption, $s$ is not in $C_P(U)$, so it is not in $R_0$. Hence, $\bar{s}$
is not in $\bar{N_{R_0}(P_0)}$. Thus, 
\begin{eqnarray}
\label{e:splustau}
\bar{s} \in \bar{N_{R_0}(P_0)}\bar{\tau}. 
\end{eqnarray}

Write $\bar{s} = \bar{[t_1,t_2,t_3]}\bar{\tau}$, where each $t_i \in R_i$, and
where we take $t_i = 1$ if $t_i \in P_i$ and $t_i = a^{2^{l-1}}$ if $t_i \notin
P_i$. As $\tau$ acts by swapping $R_1$ and $R_2$ and centralizing $R_3$, it
follows from \eqref{e:splustau} that $C_{\bar{N_{R_0}(P_0)}}(\bar{s})$ is of
order 4 generated by $\bar{[a^{2^{l-1}}, a^{2^{l-1}}, 1]}$ and
$\bar{[1,1,a^{2^{l-1}}]}$.  

Note that $t_1 = t_2$ since $\bar{s}$ is of order $2$.  We claim that
\eqref{e:Prad} and \eqref{e:splustau} imply $t_3 = 1$. Assume on the contrary
that $t_3 = a^{2^{l-1}}$. Then $O_{2'}(C_{\bar{M}}(\bar{s})) =
C_{O_{2'}(\bar{M})}(\bar{s}) \leq \bar{N_{L_1}(Q_1)}\bar{N_{L_2}(Q_2)}$.  By
\eqref{e:splustau}, $C_{O_{2',2,2'}(\bar{M})}(\bar{s}) =
C_{O_{2',2}(\bar{M})}(\bar{s})$. So since $\gen{\bar{[1,1,a^{2^{l-1}}]}}$ is a
normal $2$-subgroup of $\bar{N_S(P_0)}$, it follows that
$\gen{\bar{[1,1,a^{2^{l-1}}]}}$ is a normal $2$-subgroup of
$C_{\bar{M}}(\bar{s})$. By \eqref{e:Prad}, $\bar{s} = \bar{[1,1,a^{2^{l-1}}]}$.
This contradicts \eqref{e:splustau}. 

We conclude that $\bar{s} = \bar{\tau}$ or $\bar{s} = \bar{[a^{2^{l-1}},
a^{2^{l-1}}, 1]}\bar{\tau}$, 
and hence
\[
C_{\bar{M}}(\bar{s})  = \gen{\bar{s}} \times
C_{\bar{N_{L_1}(Q_1)N_{L_2}(Q_2)}}(\bar{s}) \times
C_{\bar{N_{L_3}(Q_3)}}(\bar{s}) \cong  C_2 \times S_3 \times S_3.
\]
However, since the two possibilities for $\bar{s}$ are conjugate under
$\bar{N_{R_1}(Q_1)}$, we may take $\bar{s} = \bar{\tau}$, as desired.
\newline\newline 

\textit{Case 2.} Assume $P_0 = Q_1Q_2Q_3'$. In this case $\bar{M} \cong S_3 \wr
\langle \bar{\tau} \rangle \times S_3$.  Replacing all occurrences of $Q_3$ by
$Q_3'$, we argue verbatim as in Case 1, except that in the present situation we
have $O_{2',2,2'}(\bar{M}) = O_{2',2}(\bar{M}) = \bar{M}$, obviating the need
to observe that $C_{O_{2',2,2'}(\bar{M})}(\bar{s}) =
C_{O_{2',2}(\bar{M})}(\bar{s})$.  Again we may take $\bar{s}=\bar{\tau}$. 
\mbox{}\newline\newline 
\textit{Case 3.} Assume $P_0 = Q_1Q_2R_3$. We have $\bar{M} = \bar{N_{L_0}(P_0)}\bar{X} \cong S_3 \wr \langle \bar{\tau} \rangle$,
$\bar{N_{R_0}(P_0)}^+ = O_{2}(M^+)$, and $\bar{N_{S}(P_0)} =
\bar{N_{R_0}(P_0)}\gen{\bar{\tau}}$. 
By assumption, $s$ is not in $C_P(U)$, so it is not in $R_0$. Hence, $\bar{s}$
is not in $\bar{N_{R_0}(P_0)}$. Thus, $\bar{s} \in \bar{N_{R_0}(P_0)}\bar{\tau}.$ 
As in the previous cases,
\eqref{e:Prad} forces $\bar{s} = \bar{\tau}$ or $\bar{[a^{2^{l-1}},
a^{2^{l-1}},1]}\bar{\tau},$ so that $$C_{\bar{M}}(\bar{s}) = \langle
\bar{s}\rangle \times C_{\bar{N_{L_1}(Q_1)N_{L_2}(Q_2)}}(\bar{s}) = C_2 \times
S_3. $$ Again, these possibilities for $\bar{s}$ are conjugate under
$\bar{N_{R_1}(Q_1)}$ and we see that we may take $\bar{s} = \bar{\tau}$, as
needed. \newline \newline \textit{Case 4.} Assume $P_0 = Q_1Q_2'R_3$. This
time, replace $\tau$ by $\tau'$, and repeat the argument from the previous
case.
\end{proof}

\begin{Rem}\label{r:102}
The minor omission in the proof of \cite[Lemma~10.2]{AschbacherChermak2010}
alluded to in the proof of Proposition~\ref{p:10.2} occurs in the middle of
page 953 with the claim``$|O_3(\bar{N})| = 9$''. It is possible that
$|O_3(\bar{N})| = 3$ under the hypotheses there. More precisely, consider
the subgroup $P = P_{0}\gen{s}$, where $P_0 = Q_1Q_2Q_3$ and $s =
[a^{2^{l-1}},1,1]\tau$.
Then $\bar{s}$ centralizes $O_3(\bar{N_{L_3}(Q_3)})$, $\bar{s}$ has order $4$,
and $\bar{s}$ squares to $\bar{[a^{2^{l-1}}, a^{2^{l-1}}, 1]}$. As $\bar{s}$
centralizes no nontrivial element in $O_{3}(\bar{N_{L_1L_2}(P_0)}) \cong C_3
\times C_3$, we have $O_3(C_{\bar{M}}(\bar{s})) = O_{3}(\bar{N_{L_3}(Q_3)})$ is
of order $3$.  But in this case, $O_2(N_{\bar{M}}(\gen{\bar{s}})) \cong D_8$
while $\gen{\bar{s}}$ is cyclic of index $2$ in this subgroup, and so
$|O_2(\Out_\K(P))| = 2$ is generated by the image of $\bar{[a^{2^{l-1}},1,1]}$.
Thus, $P_0\gen{s}$ satisfies Proposition~\ref{p:10.2}(b)(iii)(2-3), but is not
$\K$-radical.

This example also indicates how to patch the proof of
\cite[Lemma~10.2]{AschbacherChermak2010}. Paragraph 3 of page 953 gives an
argument for the statement that if (b)(iii)(2-3) holds (in our numbering), then
$P$ is centric radical.  Follow it until line $-2$ of that paragraph. In
particular, one is reduced to the case in which $P_0 = Q_1Q_2Q_3$, and $\bar{M}
\cong S_3 \wr S_3$. The Sylow $2$-subgroup $\bar{N_S(P_0)}$ of $\bar{M}$ has
the structure $D_8 \times C_2$, and it acts on $O_3(\bar{M})$ decomposably with
nontrivial summands $O_{3}(\bar{N_{L_1L_2}(P_0)})$ and
$O_{3}(\bar{N_{L_3}(P_0)})$.  Fix any element $s \in N_S(P_0)-C_S(U)$ such that
$\bar{s}$ is of order $4$ in $\bar{M}$.  Indeed, there are exactly two
possibilities for $\gen{\bar{s}}$ and hence for $P = P_0\gen{s}$, namely
$\gen{\bar{[a^{2^{l-1}},1,1]}\bar{\tau}}$
and $\gen{\bar{[a^{2^{l-1}},1,a^{2^{l-1}}]}\bar{\tau}}$. 
The latter determines a subgroup $P = P_0\gen{s}$ that is not $\K$-radical
because it does not satisfy Proposition~\ref{p:10.2}(b)(iii)(3), while the
first determines a subgroup $P = P_0\gen{s}$ that is also not $\K$-radical
(from the previous paragraph).  Hence, we must have $\bar{s}$ is of order $2$,
i.e. it is necessary that (b)(iii)(1) also holds. 
\end{Rem}

We next determine the set  $\H^{cr}$ up to $\H$-conjugacy in the case $l > 0$.

\begin{Prop}\label{p:h}
Suppose that $l > 0$. There are eighteen $\H$-conjugacy classes of elements of
$\H^{cr}$. Representatives for these classes together with their outer
automorphism groups in $\H$ are listed in Table \ref{t:3}.
\end{Prop}

\begin{table}[h]
\centering
\renewcommand{\arraystretch}{1.5}
\caption{$\H$-conjugacy classes of $\H$-centric radical subgroups, $l > 0$}
\label{t:3}
\begin{tabular}{|c|c|c|}
\hline
$P$                             & $|P|$                          & $\Out_\H(P)$                                            \\ \hline
$S$                             & $2^{10+3l}$                                  & $1$                                            \\ \hline
$Q_1Q_2Q_3$                     & $2^8$ & $S_3 \times S_3 \wr C_2$                             \\ \hline
$Q_1Q_2Q_3'$                    & $2^8$    & $S_3 \wr C_2 \times S_3$                \\ \hline
$Q_1'Q_2Q_3$                    & $2^8$    & $S_3^3$                \\ \hline
$Q_1Q_2R_3$                     & $2^{8+l}$                  & $S_3 \wr C_2$                     \\ \hline
$Q_1R_2Q_3$                     & $2^{8+l}$                  & $S_3 \times S_3$                    \\ \hline
$Q_1Q_2'R_3$                    & $2^{8+l}$                   & $S_3 \wr C_2$                         \\ \hline
$Q_1R_2Q_3'$                     & $2^{8+l}$                  & $S_3 \times S_3$                       \\ \hline

$Q_1Q_2Q_3\langle\tau\rangle$   & $2^9$                              & $S_3 \times S_3$                                    \\ \hline
$Q_1Q_2Q_3'\langle\tau\rangle$  & $2^9$                             & $S_3 \times S_3$                                   \\ \hline
$Q_1Q_2R_3\langle\tau\rangle$   & $2^{9+l}$                             & $S_3$                                               \\ \hline
$Q_1Q_2'R_3\langle\tau'\rangle$ & $2^{9+l}$                           & $S_3$                                      \\ \hline
$R_1R_2Q_3\langle\tau\rangle$   & $2^{9+2l}$                            & $S_3$                                         \\ \hline
$Q_1R_2R_3$				  &  $2^{8+2l}$                                  & $S_3$                        \\ \hline
$R_{1^7}$                                                  &  $2^7$                                     & $S_7$                                                  \\ \hline
$R_{1^7}'$                                                &  $2^6$                                     & $S_6$                             \\ \hline
$R_{1^52}$						  &  $2^{7+l}$                                  & $S_5$                          \\ \hline
$C_S(E/Z)$					  &  $2^{9+3l}$                                 & $S_3$                    \\ \hline
\end{tabular}
\end{table} 

\begin{proof}
Let $P \in \H^{cr}$.  If $N_H(P) \nleq K$ then using
Proposition~\ref{p:ngpnleqk}(a)-(d), we obtain the groups in the last 4 rows of
Table \ref{t:3}. Hence, for the remainder of the proof, we may assume that 
\begin{eqnarray}
\label{e:NHPleqK}
N_H(P) \leq K. 
\end{eqnarray}
By \cite[Lemma~3.3(a)]{LeviOliver2002}, $P$ is $\F$-centric, so that $P$ is
also $\K$-centric.

Suppose first that $P$ is $\K$ radical, so that $P \in \K^{cr}$. In this case
we appeal to Proposition~\ref{p:k} to obtain the first thirteen entries in
Table~\ref{t:3}, as follows. A case-by-case check shows that for each
$\K$-conjugacy class $\C = Y^\K$ of a subgroup $Y$ listed in Table~\ref{t:4},
one of the following holds: either no member of $\C$ is $\H$-radical ($Y =
C_S(U)$), or $\C$ meets exactly one $\H$-radical conjugacy class ($Y = S,
Q_1Q_2Q_3, Q_1Q_2Q_3\gen{\tau}, Q_1Q_2Q_3'\gen{\tau}, Q_1Q_2R_3\gen{\tau},
Q_1Q_2'R_3\gen{\tau'}$, or $R_1R_2Q_3\gen{\tau}$), or $\C$ is the class of one
of the entries in rows $4$ through $6$ of Table~\ref{t:4} ($Y = Q_1Q_2Q_3',
Q_1Q_2R_3$, or $Q_1Q_2'R_3$). In this last case, $\C$ meets one of two
$\H$-classes of $\H$-radical subgroups, and corresponding representatives of
these $\H$-classes appear in rows $3$ through $8$ of Table~\ref{t:3}. In
each of the three cases, $\Out_\H(P)$ is computed using $\eqref{e:NHPleqK}$ and
the descriptions $K = C_H(U)X$ and $H \cap K = C_K(Z) = N_H(U) =
C_H(U)\gen{\tau} = C_H(U)\gen{\tau'}$.

We illustrate the argument of the previous paragraph with four examples. First,
consider $Y = C_S(U)$. As $C_H(U)$ is normal in $K$ with Sylow $C_S(U)$, we
have $C_S(U)$ is strongly $\K$-closed. In particular, $C_S(U)^\K = \{C_S(U)\}$.
Appealing to Lemma~\ref{l:csu}, we see that $\Out_\H(C_S(U)) \cong C_2$, so
$C_S(U)$ is not $\H$-radical. 

Next, consider $Y = Q_1Q_2Q_3$. Since $X$ normalizes $Y$, we have that $Y^\K =
Y^\H$, and $N_H(Y)$ is of index $3$ in $N_K(Y)$. Since $N_H(U) =
C_H(U)\gen{\tau} = C_H(U)\gen{\tau'}$, it follows that $\Out_\H(Y) \cong S_3
\wr C_2 \times S_3$ is index $3$ in $\Out_\K(Y) \cong S_3 \wr S_3$. 

Third, consider $Y = Q_1Q_2'R_3\gen{\tau'}$. This time, no element of
$K-C_K(Z)$ normalizes $Y$, so $N_H(Y) = N_K(Y)$ by \eqref{e:NHPleqK}. However,
we claim that $Y^\K = Y^\H$.  For the proof, assume on the contrary that there
is $k \in K$ with $Y^k \leq S$ and $Y^k$ not $H$-conjugate to $Y$. Then $k
\notin K-N_H(U)$, and $k$ normalizes $R_0 = L_0 \cap S$, so $\gen{k}$ permutes
the set $\{R_1,R_2,R_3\}$ transitively. It follows that the element $\tau'^k
\in S$ interchanges $R_3$ and some other $R_i$ by conjugation. This is a
contradiction, as $R_3$ is a fixed point in the action of $S$ on
$\{R_1,R_2,R_3\}$ (see Section \ref{s:conssol}).  Hence, $Y^\K$ meets exactly
one $\H$-conjugacy class as claimed, and $\Out_\H(Y) = \Out_\K(Y)$, so $Y$ is
$\H$-radical.  

Finally, consider $Y = Q_1Q_2R_3$. Then $N_X(Y) = \gen{\tau}$ is of order $2$.
As $L_3$ is $C_H(U)$-invariant, $R_3$ is strongly closed in $C_S(U)$ with
respect to $C_H(U)$, so no element of $C_H(U)O_3(X)-C_H(U)$ normalizes $Y$.  It
follows that $N_H(Y) = N_K(Y)$ from \eqref{e:NHPleqK}. This also shows that if
we fix a nontrivial element $x \in O_3(X)$, then representatives for the
$\H$-conjugacy classes in $Y^\K$ may be taken as a subset of $\{Y, Y^x,
Y^{x^2}\} = \{Q_1Q_2R_3, Q_1R_2Q_3, R_1Q_2Q_3\}$.  As $\tau \in S$ interchanges
$R_1Q_2Q_3$ and $Q_1R_2Q_3$, it follows that $Y^\K$ meets two $\H$-conjugacy
classes (at most), with representatives $Y$ and $Q_1R_2Q_3$. From $N_H(Y) =
N_K(Y)$, we have $\Out_\H(Y) = \Out_\K(Y) \cong S_3 \wr C_2$, while
$\Out_\H(Q_1R_2Q_3) \cong S_3 \times S_3$ is induced by
$N_{L_1}(Q_1)N_{L_3}(Q_3)$. So indeed $Q_1R_2Q_3$ is $\H$-radical and not
$\H$-conjugate to $Y$.

It remains to consider the case in which $P$ is not $\K$-radical.  We claim
here that $P$ is $\H$-conjugate to $Q_1R_2R_3$, the last remaining entry of
Table~\ref{t:3}.  Observe first that $P \leq C_S(U) = R_0\gen{\mathbf{d}}$.
Indeed, otherwise $Z(P) \cap U = Z$ would be $N_K(P)$-invariant, and so as
$N_H(P) = N_K(P)$ by \eqref{e:NHPleqK}, this would yield that $\Out_\H(P) =
\Out_\K(P)$ has no nontrivial normal $2$-subgroups, contradicting the
assumption that $P$ is not $\K$-radical. Hence, $P \leq C_S(U) =
R_0\gen{\mathbf{d}}$ as claimed, and so $U \leq Z(P)$ since $P$ is centric. 

Set $P_0 = P \cap L_0$, and for each $i = 1,2,3$, let $P_i$ be the
projection of $P_0$ in $R_i$ as before; see the remarks just before
Proposition~\ref{p:10.2}. A reading of the first three paragraphs of the proof
of \cite[Lemma~10.2]{AschbacherChermak2010} reveals that the given argument
applies to an $\H$-centric radical $P \leq S$ whose normalizer $N_H(P)$ is
contained in $K$, our current situation \eqref{e:NHPleqK}. We conclude that
$P_i = P \cap R_i$ and that $P_i \in \Q_i$ or $P_i = R_i$ for each $i = 1,2,3$.
In particular, $P_0 = P_1P_2P_3$. 

We next claim that $P = P_0$. Suppose on the contrary that $P_0 < P = C_P(U)$,
and choose $d \in P-C_P(U)$. Then $d \in R_0\gen{\mathbf{d}}-R_0$, and since
$\mathbf{d}$ interchanges the $R_0$-conjugacy classes of subgroups in $\Q_i$
for each $i$ (c.f. Notation~\ref{N:Knotation2}(e,f)), $d$ has the same
property.  On the other hand, as $d$ normalizes $R_i$ and $P$, it normalizes
$P_i = P \cap R_i$. We conclude that $P_i = R_i$ for each $i$. But then $P =
R_0\gen{\mathbf{d}} = C_S(U)$ and $\Out_\H(P)$ is of order $2$. Thus, $P$ is
not $\H$-radical, contrary to the original choice of $P$.

Finally, conjugating in $L_0\gen{\mathbf{d}} = C_H(U) \leq H$ if necessary, we
have $P = Q_1R_2R_3$ or $R_1R_2Q_3$. But in the latter case, $\Out_\H(P) \cong
C_2 \times S_3$ is induced by $\gen{\tau} \times N_{L_3}(Q_3)$, so again, $P$
is not $\H$-radical, a contradiction. Thus, up to $\H$-conjugacy, $P =
Q_1R_2R_3$ and $\Out_\H(P) \cong S_3$ is induced by $N_{L_1}(Q_1)$, and this is
the only remaining entry in Table~\ref{t:3}.
\end{proof}

Finally, we are able to describe the set of $\F$-centric radical subgroups, up
to $\F$-conjugacy:
\begin{Thm}
Let $\F=\Sol(5^{2^l})$ with $l > 0$. Representatives for the $\F$-conjugacy classes of
$\F$-centric radical subgroups, together with their orders and automorphism
groups, are listed in Table \ref{t:2}, where `$-$' indicates that the subgroup
is not centric radical in that fusion system.
\end{Thm}

\begin{table}[h]
\centering
\renewcommand{\arraystretch}{1.5}
\caption{$\F$-conjugacy classes of $\F$-centric radical subgroups, $l > 0$ }
\label{t:2}
\begin{tabular}{|c|c|c|c|c|}
\hline
$P$                             & $|P|$    & $\Out_\K(P)$                          & $\Out_\H(P)$                         & $\Out_\F(P)$                       \\ \hline
$S$                             & $2^{10+3l}$ & $1$                                  & $1$                                  & $1$                  \\ \hline
$C_S(U)$                        & $2^{9+3l}$ & $S_3$                            & $-$                                  & $S_3$                       \\ \hline
$Q_1Q_2Q_3$                     & $2^8$    & $S_3 \wr S_3$                & $S_3 \times S_3 \wr C_2$ & $S_3 \wr S_3$             \\ \hline
$Q_1'Q_2Q_3$                    & $2^8$    & $S_3 \times S_3 \wr C_2$ & $S_3^3$                          & $S_3 \times S_3 \wr C_2$ \\ \hline
$Q_1Q_2R_3$                     & $2^{8+l}$    & $S_3 \wr C_2$                & $S_3 \wr C_2$                & $S_3 \wr C_2$           \\ \hline
$Q_1'Q_2R_3$                    & $2^{8+l}$    & $S_3 \wr C_2$                & $S_3 \wr C_2$                & $S_3 \wr C_2$               \\ \hline
$Q_1Q_2Q_3\langle\tau\rangle$   & $2^9$    & $S_3 \times S_3$                          & $S_3 \times S_3$                          & $S_3 \times S_3$             \\ \hline
$Q_1Q_2Q_3'\langle\tau\rangle$  & $2^9$    & $S_3 \times S_3$                          & $S_3 \times S_3$                          & $S_3 \times S_3$              \\ \hline
$Q_1Q_2R_3\langle\tau\rangle$   & $2^{9+l}$ & $S_3$                            & $S_3$                            & $S_3$                     \\ \hline
$Q_1Q_2'R_3\langle\tau'\rangle$ & $2^{9+l}$ & $S_3$                            & $S_3$                            & $S_3$                        \\ \hline
$R_1R_2Q_3\langle\tau\rangle$   & $2^{9+2l}$ & $S_3$                            & $S_3$                            & $S_3$                    \\ \hline
$R_{1^7}$                           &  $2^7$        & $-$                               &$S_7$                            & $S_7$                        \\ \hline
$R_{1^7}'$                         &  $2^6$        & $-$                               & $S_6$                            & $S_6$     \\ \hline
$R_{1^52}$					  &  $2^{7+l}$    & $-$                               & $S_5$                            & $S_5$\\ \hline
$C_S(E/Z)$					  &  $2^{9+3l}$    & $-$                               & $S_3$                            & $S_3$ \\ \hline
$C_S(E)$            &   $2^{7+3l}$       & $-$                                  & $-$                                  &    $\GL_3(2)$         \\ \hline
$A$                           &  $2^4$         & $-$                                  & $-$                                  &     $\GL_4(2)$                       \\ \hline
\end{tabular}
\end{table}

\begin{proof}
This follows upon combining Propositions \ref{p:fcr}, \ref{p:k} and \ref{p:h}.
Note that $Q_1R_2R_3$ appears in Table~\ref{t:3}, but it does not appear in
Table~\ref{t:2} because it does not satisfy the hypotheses of
Proposition~\ref{p:fcr}(c)(ii): there is an involution in $K \leq G$ of the
form $\tau^x$ for a nonidentity $x \in O_3(X)$ which normalizes $Q_1R_2R_3$ but
is not contained in $H$. Indeed, $O_{2}(\Out_\F(Q_1R_2R_3))$ is of order $2$
and induced by conjugation by this element, so $Q_1R_2R_3$ is not $\F$-radical.
\end{proof}

\section{Number of projective simple modules}\label{s:blocks}

In this section we calculate the number of projective simple modules for the
outer automorphism groups in the various tables of the previous section. Let
$G$ be a finite group, and let $\Irr(G)$ be the set of ordinary irreducible
characters of $G$, i.e. those over a splitting field of characteristic $0$.  A
character $\chi \in \Irr(G)$ is said to be of \emph{$p$-defect} $d$ if
$|G|_p/\chi(1)_p = p^d$, where $n_p$ denotes the $p$-part of the integer $n$. 

Write $z(kG)$ for the number of projective simple $kG$-modules. Any projective
simple module is the unique indecomposable module in the block of $kG$ in which
it lies. Blocks of $kG$ of defect $0$ are exactly those which contain such a
module.  Equivalently, blocks of defect $0$ are exactly those which contain a
unique ordinary irreducible character of defect $0$. Thus, one can count
projective simple $kG$-modules by looking at the list of irreducible character
degrees for $G$. 

\begin{Thm}\label{t:pdef}
Let $G$ be a finite group and $p$ be a prime. Then $z(kG)$ is the number of
ordinary irreducible characters of defect $0$. 
\end{Thm}
\begin{proof}
See \cite[Theorem 3.18]{Navarro1998}.
\end{proof}

Along with Theorem~\ref{t:pdef}, the ``method of little groups'' will be used
in order to compute the character degrees of various solvable groups.

\begin{Thm}[Method of little groups]
\label{t:litgrp}
Let $G$ be a finite group with normal subgroup $A$, and let $\Theta =
[\Irr(A)/G]$ denote a set of representatives for the $G$-orbits on $\Irr(A)$.
Assume that each irreducible character $\theta$ of $A$ extends to an
irreducible character $\hat{\theta}$ of its inertia subgroup $I_G(\theta)$.
Then there is a bijection 
\begin{eqnarray}
\label{e:litgrp}
\{\,\,(\theta, \beta) \,\, \mid \,\, \theta \in \Theta, \,\,\beta \in \Irr(I_G(\theta)/A)\,\,\} \longrightarrow \Irr(G)
\end{eqnarray}
given by sending $(\theta, \beta)$ to the induced character
$(\hat{\theta}\beta)\!\uparrow^G_{I_G(\theta)}$, where $\hat{\theta}$ is any
extension of $\theta$, and where $\beta$ is regarded also as an irreducible character of
$I_G(\theta)$ with $A$ in its kernel.
\end{Thm}
\begin{proof}
This is a standard result which is a consequence of
\cite[11.5]{CurtisReiner1990}. Fix as in the statement a set $\Theta$ of
representatives of the $G$-orbits on $\Irr(A)$, and for each $\theta \in
\Theta$, an extension $\hat{\theta} \in \Irr(I_G(\theta))$.  For each pair
$(\theta, \beta)$ as in \eqref{e:litgrp}, the character
$(\hat{\theta}\beta)\!\uparrow^G_{I_G(\theta)}$ is irreducible by
\cite[11.5(ii)]{CurtisReiner1990}. So the map in \eqref{e:litgrp} is
well-defined. 

Assume $(\hat{\theta}_1\beta_1)\!\!\uparrow^G_{I_G(\theta_1)} =
(\hat{\theta}_2\beta_2)\!\!\uparrow_{I_G(\theta_2)}^G$. The restriction of this
character to $I_{G}(\theta_i)$ has $\hat{\theta}_i\beta_i$ as a constituent by
Frobenius reciprocity, and so the restriction to $A$ has $\theta_i$ as a
constituent since $A \leq \ker(\beta_i)$.  It follows that $\theta_1$ and
$\theta_2$ are $G$-conjugate by Clifford's Theorem. Hence, $\theta_1 =
\theta_2$ since they were taken in $\Theta$. By
\cite[11.5(iii)]{CurtisReiner1990}, $\beta_1 = \beta_2$. 

Let $\chi \in \Irr(G)$, and let $\theta$ be any irreducible constituent of the
restriction of $\chi$ to $A$, with chosen lift $\hat{\theta}$ to $I_G(\theta)$.
Since $G$ is transitive on such constituents by Clifford's Theorem, we may take
$\theta \in \Theta$. By Frobenius reciprocity, $\chi$ is a constituent of
$\theta\!\uparrow_{A}^{G}$, and the latter decomposes as
\[
\theta\!\uparrow_A^G \quad = \sum_{\beta \in
\Irr(I_G(\theta)/A)}\beta(1)\,\,(\hat{\theta}\beta)\!\uparrow_{I_G(\theta)}^G
\]
by \cite[11.5(iii)]{CurtisReiner1990}. Hence, $\chi =
(\hat{\theta}\beta)\!\uparrow_{I_G(\theta)}^G$ for some $\beta$. 
\end{proof}

\begin{Prop}\label{p:def0}
Each of the groups listed in Table \ref{t:def0} has the stated number of blocks
of defect zero.
\end{Prop}

\begin{table}[h]
\centering
\renewcommand{\arraystretch}{1.5}
\caption{The number of projective simple modules}
\label{t:def0}
\begin{tabular}{|c|c||c|c||c|c|}

\hline
$G$                                          & $z(kG)$ & $G$                                          & $z(kG)$ & $G$                                          & $z(kG)$  \\ \hline

$S_3$                                        & $1$   &  $(C_3 \times C_3) \overset{-1}{\rtimes} C_2$ & $4$   & $S_6$                                        & $1$     \\ \hline
$S_3 \times S_3$                                        & $1$    & $(C_3)^3 \rtimes (C_2 \times S_3)$          & $1$    & $S_3 \wr S_3$                                & $1$     \\ \hline
$S_3 \times S_3 \times S_3$                                        & $1$    & $(C_3)^3 \rtimes (C_2 \times C_2)$          & $4$   & $S_5$                                        & $0$     \\ \hline
$S_3 \wr C_2$                                & $0$    & $\GL_3(2)$                                   & $1$   & $A_7$                                        & $0$     \\ \hline
$S_3 \wr C_2 \times S_3$                     & $0$    & $\GL_4(2)$                                   & $1$   & $S_7$                                        & $0$     \\ \hline
\end{tabular}
\end{table}

\begin{proof}
Let $G$ be one of the groups listed in Table \ref{t:def0}. It suffices by
Theorem \ref{t:pdef} to compute the number of characters having degree
divisible by the $2$-part of the group order. The character tables of
$G=\GL_3(2),S_5,S_6,A_7,S_7,\GL_4(2) \cong A_8$ can be found in the ATLAS
\cite{ATLAS}. For those $G$ which split as a direct product $G = G_1 \times
G_2$, we use the fact that the irreducible characters of $G$ are the pairwise
tensor products of the irreducible characters of $G_1$ and $G_2$.  

In all remaining cases, the character degrees are computed using
Theorem~\ref{t:litgrp} by taking $A$ to be a normal elementary abelian
$3$-group which is complemented in $G$. Each irreducible character of $A$
extends to its stabilizer in $G$ in this case by
\cite[11.8(ii)]{CurtisReiner1990}, so Theorem~\ref{t:litgrp} applies. For a
representative $\theta$ of an orbit of $G/A$ on $\Irr(A)$, we compute the
$2$-parts of the index in $G$ and of the irreducible character degrees
$\beta(1)$ of the inertia subgroup $I_G(\theta)$.  The pairs $(\theta, \beta)$
with $\theta(1)_2 \cdot [G:I_G(\theta)]_2 \cdot \beta(1)_2$ equal to the
$2$-part of the group order are recorded in Table~\ref{t:chardef0}. 

For example, suppose that $G=(C_3)^3 \rtimes (C_2 \times S_3)$ and set
$A=(C_3)^3 \leq G$.  For each $\theta \in \Irr(A)$ we have $\theta=
\theta_{i_1} \otimes \theta_{i_2} \otimes \theta_{i_3}$ for some $1 \le i_j \le
3$. An $S_3$ factor in  $G/A \cong S_3 \times C_2$ acts on $\Irr(A)$ by
permuting the $\theta_{i_j}$ while the $C_2$ factor fixes $\theta_1$ and
interchanges $\theta_2$ and $\theta_3$ in each coordinate. One computes that
there are six orbits on irreducible characters.  The pair $(\theta, \beta)$,
where $\theta = \theta_2 \otimes \theta_2 \otimes \theta_2$ and $\beta$ is the
degree $2$ irreducible character of $I_G(\theta)/A \cong S_3$ gives rise
to the only irreducible character of $G$ of $2$-defect $0$. The remaining cases
are summarized in Table~\ref{t:chardef0}, where a representative
$\theta_{i_1}\otimes \theta_{i_2} \otimes \theta_{i_3}$ is abbreviated to
$[i_1,i_2,i_3]$, for example. 
\begin{table}[h]
\centering
\renewcommand{\arraystretch}{1.5}
\caption{Characters $\chi \in \Irr(G)$ of defect $0$}
\label{t:chardef0}
\begin{tabular}{|c|c|c|c|c|c|}
\hline
$G$ & $A$ & $\theta$ & $I_G(\theta)/A$ & $\beta(1)$ & $\chi(1)$\\
\hline
\hline
$S_3$ & $O_3(G)$ & $[2]$ & $1$ & $1$ & $2$ \\
\hline
$S_3 \wr C_2$ & $O_3(G)$ & -- & -- & -- & -- \\
\hline
$S_3 \wr S_3$ & $O_3(G)$ & $[2,2,2]$ & $S_3$ & $2$ & $16$\\ 
\hline 
$(C_3)^2 \overset{-1}{\rtimes} C_2$ & $O_3(G)$ & $[1,1], [1,2], [2,1], [2,2]$ & $1,1,1,1$ & $1,1,1,1$ & $2,2,2,2$\\
\hline
$(C_3)^3 \overset{-1, (1,2)}{\rtimes} (C_2 \times C_2)$ & $O_3(G)$ & $[1,2,1], [1,2,2], [1,2,3], [2,3,2]$ & $1,1,1,1$ & $1,1,1,1$ & $4,4,4,4$ \\
\hline
$(C_3)^3 \overset{-1, \wr}{\rtimes} (C_2 \times S_3)$ & $(C_3)^3$ & $[2,2,2]$ & $S_3$ & $2$ & $4$\\
\hline
\end{tabular}
\end{table}
\end{proof}

Using Tables \ref{t:1}, \ref{t:3} and \ref{t:2} we can give a count of the
number of weights. 

\begin{Cor}\label{c:solqh}
For $\D \in \{\H,\F\}$ and all $l \ge 0$, the number of weights associated with
the K\"ulshammer-Puig pair $(\D,0)$ is 
\[
\mathbf{w}(\D,0)=12.
\]
\end{Cor} 

Note that $\mathbf{w}(\H,0)=12$ is known as a consequence of results in \cite{An1993}. 

\section{K\"ulshammer-Puig classes}\label{s:KP}

We give here a proof of Theorem~\ref{t:solblocks} essentially by direct
computation. Throughout this section, we fix an arbitrary nonnegative integer
$l$ and set $q = 5^{2^l}$. We adopt the notation $\F$, $\H$, $\K$ from
Section~\ref{s:solq}. These systems depend implicitly on $q$. 

Recall that the Schur multiplier of a finite group $G$ is the cohomology group
$M(G) := H^2(G,\mathbb{C}^\times)$. It is a finite abelian group. Given any
algebraically closed field $k$ of characteristic $2$, the $2'$-primary part of
$M(G)$ is isomorphic to $H^2(G,k^\times)$.  The approach taken to showing
Theorem~\ref{t:solblocks} requires the explicit computation of
$H^2(G,k^\times)$ (the values of the functor $\A^2$) for each group $G$
appearing as the outer automorphism group of a centric radical in
Section~\ref{s:solqcrel}. The computation of Schur multipliers of finite groups
is typically a delicate task. In our case the task is simpler for two reasons.
First, the outer automorphism groups are relatively small finite groups.
Second, the task is simpler because of the following lemma, which allows us in
many cases to reduce the computation of the odd part of the Schur multiplier to
computations of $H^2(G,\mathbb{F}_p)$ for odd primes $p$. A finite group is
said to be $p$-\emph{perfect} if it has no nontrivial $p$-group as a quotient.

\begin{Lem}
\label{lem:oddperfect}
Let $\D \in \{\F,\H,\K\}$. Then for each subgroup $P \in \D^{rc}$ and each odd
prime $p$, the outer automorphism group $\Out_\D(P)$ is $p$-perfect. 
\end{Lem}
\begin{proof}
Direct inspection of the outer automorphism groups in Tables~\ref{t:1} and
\ref{t:2}.
\end{proof}

The next lemma collects various standard results on group cohomology stated in
the special cases in which they will be used. We thank the referee for several
simplifications in our original arguments.

\begin{Lem}\label{lem:coh}
Let $G$ and $H$ be finite groups and let $k$ be an algebraically closed field
of characteristic $2$. Write $|G| = 2^rw$ where $w$ is odd. The following hold.
\begin{enumerate}
\item[(a)] For any abelian group $A$ with trivial $G$-action, $H^1(G,A) =
\Hom(G,A)$.
\item[(b)] There is a surjective map
\[
H^2(G, \mathbb{Z}/w\mathbb{Z}) \to H^2(G,k^\times), 
\]
which is an isomorphism if and only if $G$ is $p$-perfect for every odd
prime $p$ dividing $|G|$.
\item[(c)]  If $H^2(G,\mathbb{F}_p) = 0$ for all odd primes $p$, then
$H^2(G,k^\times) = 0$. 
\item[(d)]  If $p$ is odd and $G$ is a $p$-perfect group with cyclic Sylow $p$-subgroups then $H^2(G,\FF_p) = 0$.
\item[(e)]  If $G$ is a $p$-perfect group with an elementary abelian Sylow $p$-subgroup $V$ of order $p^2$, then 
\[
H^2(G,\FF_p) =
\begin{cases}
\FF_p & \text{ if $\Aut_G(V) \subseteq \SL(V)$, and}\\
0 & \text{ otherwise.}
\end{cases}
\]
\item[(f)]  If $G$ and $H$ are $p$-perfect, then 
\[
H^2(G \times H, \mathbb{F}_p) \cong H^2(G,\mathbb{F}_p) \oplus H^2(H, \mathbb{F}_p).
\]
\item[(g)] Let $p$ be an odd prime. If $G$ is $p$-perfect and the
$p$-part $M(G)_{p}$ of the Schur multiplier of $G$ is of exponent at most $p$
then
\[
M(G)_{p} = H^2(G,\CC^{\times}) \otimes \ZZ_{(p)} \cong H^2(G,\FF_p) \cong
H^2(G,k^{\times}) \otimes \ZZ_{(p)}. 
\]
Here, $\ZZ_{(p)}$ denotes the $p$-local integers.
\item[(h)] (Schur) If $M(G)$ has exponent $e$, then $e^2$ divides the order of
$G$. 
\end{enumerate}
\end{Lem}
\begin{proof}
\begin{enumerate}
\item[(a)] This follows from the description $H^1(G,A)$ as the group of
derivations $G \to A$ \cite[Corollary~6.4.6]{WeibelHA}.
\item[(b)] Fix a Sylow $2$-subgroup $S$ of $G$.  Since $k^{\times}$ has all odd
roots of unity, powering by $w$ is a surjective endomorphism with kernel
$\mathbb{Z}/w\mathbb{Z}$. Thus, there is an exact sequence
\[
H^1(G,k^{\times}) \to H^2(G,\mathbb{Z}/w\mathbb{Z}) \to H^2(G,k^\times) \to H^2(G,k^{\times}).
\]
The last map is multiplication by $w = |G:S|$, and so it factors as
\[
H^2(G,k^{\times}) \xrightarrow{\res} H^2(S,k^{\times}) \xrightarrow{\tr} H^2(G,k^{\times})
\]
by \cite[Proposition~3.6.17]{Benson1998b} applied with $M = M' =
k^{\times}$.  Since $H^2(S,k^{\times}) = 0$, we conclude that the last map is
$0$. The middle map is therefore a surjection, and since $H^1(G,k^\times) =
\Hom(G,k^{\times})$ by (a), we see that it is an isomorphism if $G$ is
$p$-perfect for every odd prime $p$ dividing $|G|$. 
\item[(c)]  This follows upon filtering $\ZZ/w\ZZ$ by subgroups of prime
order, considering the corresponding long exact sequences in cohomology, and
applying (b).
\item[(d)] Let $P$ be a Sylow $p$-subgroup with $p$ odd. Restriction induces an
isomorphism $H^*(G,\FF_p) \to H^*(N_G(P),\FF_p) \cong
H^{*}(P,\FF_p)^{\Aut_G(P)}$ by \cite[Corollary~3.6.19]{Benson1998b} applied
with $M = M' = \mathbb{F}_p$. Now $H^*(P,\FF_p) \cong \FF_p[x,y]/(x^2)$ with
$\deg x = 1$ and $\deg y = 2$ \cite[Proposition~3.5.5]{Benson1998b}, and the
Bockstein $H^1(P,\FF_p) \to H^2(P,\FF_p)$ is an isomorphism of $N_G(P)$-modules
(cf. \cite[p.132, Example]{Benson1998b}). As $N_G(P)$ has no invariants in
$H^1(P,\FF_p)$ by assumption it also has no invariants on $H^2(P,\FF_p)$. 

\item[(e)] Restriction to $V$ again identifies $H^*(G,\FF_p)$ with the
invariants $H^*(V,\FF_p)^{\Aut_G(V)}$.  Now 
\[
H^{*}(V,\FF_p) \cong \Lambda_{\FF_p}(x_1,x_2) \otimes \FF_p[y_1,y_2],
\]
with $\deg x_i = 1$ and $\deg y_i = 2$ by
\cite[Corollary~3.5.7(ii)]{Benson1998b}, so that
\[
H^1(V,\FF_p) = \gen{x_1,x_2}_{\FF_p} \quad \text{ and } \quad H^2(V,\FF_p) =
\gen{y_1,y_2, x_1x_2}_{\FF_p}. 
\]
Here, $H^1(V,\FF_p)$ is the natural module for $\Aut(V) \cong \GL(V)$, while
$H^2(V,\FF_p)$ is the direct sum of the natural module $\gen{y_1,y_2}_{\FF_p}$
and $\gen{x_1x_2}_{\FF_p}$ on which $\GL(V)$ acts via the determinant map.  By
assumption, $\Aut_G(V) \leq \GL(V)$ has no fixed points on the natural module,
so $H^2(V,\FF_p)$ is nontrivial generated by $x_1x_2$ if and only if every
element of $\Aut_G(V)$ has determinant 1.
\item[(f)]  This follows from the K\"unneth Theorem \cite[Theorem
2.7.1]{Benson1998} and the assumption. 
\item[(g)] Powering by $p$ on $\CC^{\times}$ gives the exact sequence
\[
H^1(G,\CC^\times) \to H^2(G,\FF_p) \to H^2(G,\CC^\times) \xrightarrow{p} H^2(G,\CC^{\times}). 
\]
The assumptions imply that tensoring with $\ZZ_{(p)}$ kills $H^1(G,\CC^\times)$
and the last map.  Hence, $H^2(G,\FF_{p}) \cong H^2(G,\CC^\times) \otimes
\ZZ_{(p)}$. Since $M_{2'}(G)$ is isomorphic to $H^{2}(G,k^\times)$
(\cite[Proposition 2.1.14]{Karpilovsky1987}) and $p$ is odd, the $p$-primary part of $H^2(G,k^\times)$ is
of exponent at most $p$ by assumption. Thus, the exact same argument with
$k^{\times}$ in place of $\CC^{\times}$ shows that $H^2(G,\FF_p) \cong
H^2(G,k^{\times}) \otimes \ZZ_{(p)}$.
\item[(h)] We refer to \cite[2.1.5]{Karpilovsky1987} for a proof.
\end{enumerate}
\end{proof}

We are interested in computing the cohomology of the functor $H^2(-,k^\times)$
defined on the subdivision category of the full subcategory of the fusion
systems $\F$, $\H$, and $\K$, respectively, on the collection of centric
subgroups.  

Let $\C$ be any full subcategory of a saturated fusion system.
Recall from \cite[Definition~8.13.2]{LinckelmannBT2} that the \emph{subdivision
category} $S(\C)$ of proper inclusions is the category with objects the chains
$\sigma = (X_0 < X_1 < \cdots < X_m)$ (of proper inclusions) in $\C$; here $m$
is the \emph{length} $|\sigma|$ of $\sigma$. Given another object $\tau = (Y_0
< Y_1 < \cdots < Y_n) \in S(\C)$, a morphism from $\sigma$ to $\tau$ consists
of an order preserving function $\beta\colon \{0,1\dots,m\} \to
\{0,1,\dots,n\}$ together with isomorphisms $\phi_i \colon X_i \to
Y_{\beta(i)}$ in $\C$ for each $i \in \{0,1,\dots,m\}$ which make the evident
skewed ladder commute. In particular, the automorphism group of the chain
$\sigma$ in $\C$ may be identified with the group of automorphisms of $X_m$
which preserve $X_i$ for all $0 \leq i \leq m$.  We write $[S(\C)]$ for the
poset of isomorphism classes of objects in $S(\C)$, where $[\sigma] \leq
[\tau]$ if there are representatives $\sigma' \in [\sigma]$ and $\tau' \in
[\tau]$ and a morphism $\sigma' \to \tau'$ in $S(\C)$. 

There is a simpler resolution than the standard bar resolution for computing
cohomology of a functor defined on the subdivision category of any EI-category,
which was given in \cite{Linckelmann2005}.

\begin{Lem}\label{L:cochain}
Let $\C$ be any full subcategory of a saturated fusion system and let $F\colon
[S(\C)] \to \Ab$ be a covariant functor. The cohomology groups
$H^n([S(\C)],F)$, and thus the derived functors of $\lim F$, can be computed
via the cochain complex $C^*(F)$ defined as follows:
\[
C^n(F) = \bigoplus_{|\sigma|=n} F([\sigma]),
\]
whose elements are viewed as functions $\alpha$ from isomorphism classes of
chains of length $n$, and where $|\sigma|$ denotes the length of $\sigma$. The
coboundary map $\delta^n \colon C^n(F) \to C^{n+1}(F)$ is defined by
\[
\delta^n(\alpha)([\sigma]) =  \sum_{i = 0}^n (-1)^{i} F(\iota_{[\sigma(i)], [\sigma]})(\alpha([\sigma(i)])),
\]
where $\sigma(i)$ denotes the chain $\sigma$ with its $i$th term removed, and
$\iota_{[\sigma(i)], [\sigma]}$ denotes the unique morphism from $[\sigma(i)]$
to $[\sigma]$.
\end{Lem}
\begin{proof}
This is \cite[Proposition~3.2]{Linckelmann2005}, applied as in
\cite[Lemma~3.1]{Park2010}.
\end{proof}

Finally, the aim of the following, highly specialized lemma is to orient
the reader to the way in which Lemma~\ref{L:cochain} will be used later in the
proof of Theorem~\ref{t:lim0}.

\begin{Lem}\label{L:limposet}
Let $\C$ be a saturated fusion system, and let $F \colon [S(\C^{cr})] \to \Ab$
be a covariant functor. Then $\Lim{[S(\C^{cr})]}F = 0$ under either of the
following conditions.
\begin{enumerate}
\item[(a)] $F([X]) = 0$ for all subgroups $X \in \C^{cr}$;
\item[(b)] $F$ is zero on all but two distinct chains $[X_0]$ and $[X_1]$ of
length zero, and there exists a subgroup $Y \in \C^{cr}$ such that $F([Y])= 0$,
\begin{enumerate}
\item[(i)] $X_0 < X_1 > Y$, and
\item[(ii)] the maps $F([X_0]) \to F([X_0 < X_1])$ and $F([X_1]) \to F([Y
< X_1])$ are injective.
\end{enumerate}
\end{enumerate}
\end{Lem}
\begin{proof}
We view $\Lim{[S(\C^{cr})]} F$ as the degree $0$ cohomology of the functor $F$.
As such it can be computed by using the cochain complex of
Lemma~\ref{L:cochain}.  The coboundary map $\delta^0 \colon C^0(F) \to C^1(F)$
on $0$-cochains is obtained by extending linearly from
\[
\delta^0(\alpha)([X < X']) = F(\iota_{[X'],[X<X']})(\alpha([X'])) - F(\iota_{[X],[X<X']})(\alpha([X])).
\]

With this in mind, the two parts of the lemma are simply ways of saying that
the kernel of $\delta^0$, and thus $\Lim{[S(\C^{cr})]} F$, is $0$.  This is
trivial in the case of part (a). The assumption in (b) implies that $C^0(F) =
F([X_0]) \oplus F([X_1])$ and then (i) and (ii) ensure that the composite
\[
C^0(F) \xrightarrow{\delta^0} C^1(F) \xrightarrow{\operatorname{proj}} F([X_0 <
X_1]) \oplus F([Y < X_1])
\]
is injective.
\end{proof}

We now begin the computation of the higher limits of $H^2(-,k^\times)$ in
the cases of interest.

\begin{Lem}\label{l:trivmult}
Fix an algebraically closed field $k$ of characteristic $2$, and let $\D
\in \{\F,\H,\K\}$. For each $P \in \D^{cr}$, one of the following holds.
\begin{enumerate}
\item[(a)] $H^2(\Out_\D(P), k^\times) = 0$, or
\item[(b)] $l = 0$, $H^2(\Out_\D(P), k^\times) \cong H^2(\Out_\D(P), \FF_3) \cong C_3$, and either
\begin{enumerate}
\item[(i)] $P = QR_{1^7}$, or
\item[(ii)] $P = Q$ and $\D = \H$, or
\item[(iii)] $P = R_{1^7}$ and $\D = \H$ or $\F$. 
\end{enumerate}
\end{enumerate}
\end{Lem}
\begin{proof}
We first prove the lemma for $l > 0$. Fix $P \in \D^{cr}$ appearing in
Tables~\ref{t:4}, \ref{t:3}, or \ref{t:2}, and let $G = \Out_\D(P)$ be its
outer automorphism group in $\D$, for short. In order to show that (a) holds in
this case $(l > 0)$, it suffices to show that $H^2(G,\FF_p) = 0$ for all odd
primes $p$ by Lemma~\ref{lem:coh}(c). Now $G$ is $p$-perfect for all odd primes
$p$ by Lemma~\ref{lem:oddperfect}. An inspection of the tables shows that one
of three cases holds: (1) $G$ has cyclic Sylow $p$-subgroups for all odd primes
$p$, or (2) $G$ has cyclic Sylow $p$-subgroups for all $p \geq 5$ and
elementary Sylow $3$-subgroups of order $3^2$, or (3) $G \cong S_3 \wr S_3$. By
Lemma~\ref{lem:coh}(d), we have $H^2(G,\FF_p)=0$ for all odd primes $p$ in Case
(1). Assume (2). Then $G \cong S_6$, $S_7$, $GL_4(2)$, $S_3 \times S_3$, or
$S_3 \wr C_2$. In all cases, $H^2(G,\FF_p) = 0$ for all $p \geq 5$ again by
Lemma~\ref{lem:coh}(d). For a Sylow $3$-subgroup $V$ of $G$, we have $\Aut_G(V)
\nleq \SL(V)$ by direct computation, and so $H^2(G,\mathbb{F}_3) = 0$ in Case
(2) as well, by Lemma~\ref{lem:coh}(e). Finally, assume Case (3), so that $G
\cong S_3 \wr S_3$. Again, we just need to show that $H^2(G,\FF_3) = 0$.
In this case one can apply the Lyndon-Hochschild-Serre spectral sequence
\cite[6.8.2]{WeibelHA} with respect to the base $B$ of the wreath product. The
relevant parts of the $E_2$-page are
\begin{itemize}
\item $H^{0}(S_3, H^{2}(B, \mathbb{F}_3)) = 0$ (the coefficients are 0);
\item $H^{1}(S_3, H^1(B, \mathbb{F}_3)) = 0$ (since the base is $3$-perfect); and 
\item $H^2(S_3, H^0(B, \mathbb{F}_3)) = 0$ (trivial invariants). 
\end{itemize}
Hence, $H^2(G,\mathbb{F}_3) = 0$. This completes the proof in the case $l >
0$. 

We now turn to the case $l=0$. By inspection of Table~\ref{t:1}, either it was
shown in the previous case that $H^2(\Out_\D(P), k^\times) = 0$, or else the
subgroup $P$ is listed in (b)(i)-(b)(iii) of the lemma. We go through these
three cases in turn, and we set $G = \Out_\D(P)$ again for
short.\newline
\newline
\noindent
\textit{Case 1.} $P = QR_{1^7}$ and $G \cong (C_3 \times C_3)
\overset{-1}{\rtimes} \gen{\mathbf{d}}$: 

Recall that $M_{2'}(G) \cong H^2(G,k^\times)$ is an abelian group of odd
order, so it must be $M_3(G)$.  Let $e$ be its exponent. From
Lemma~\ref{lem:coh}(h), $e^2$ divides $|G| = 3^2\cdot 2$. Hence, $e = 1$ or
$3$. By Lemma~\ref{lem:coh}(g), we see that 
\[
H^2(G, \FF_3) \cong H^2(G, k^\times) \otimes \ZZ_{(3)} = H^2(G,k^\times),
\]
and $H^2(G, \FF_3) \cong \FF_3$ by Lemma~\ref{lem:coh}(e) ($\mathbf{d}$ acts by
minus the identity). This completes the proof of Case 1.
\newline \newline
\textit{Case 2.} $P = Q$: Suppose first that $\D = \H$. Then $G :=
\Out_\D(P) \cong C_3^3 \rtimes (C_2 \times C_2)$ with one factor inverting
$C_3^3$ and the other swapping the first two $C_3$ factors. We first claim
that the exponent of $H^2(G,k^\times)$ is not divisible by $3^2$. Indeed, this
follows directly from Lemma~\ref{lem:coh}(h), as otherwise $|G|$ would be
divisible by $3^4$, which is not the case. It follows that $H^2(G,k^\times)
\cong H^2(G,\FF_3)$ by Lemma~\ref{lem:coh}(g).  Now
\begin{equation}\label{e:wr}
H^2(C_3^3, \FF_3) = \gen{x_1x_2, x_1x_3, x_2x_3, y_1, y_2, y_3}_{\FF_3},
\end{equation}
where the $y_i$ are polynomial generators and the $x_i$ are exterior
generators. Further, $H^2(G, \FF_3)$ is the invariants under
$\gen{\mathbf{d},\tau}$ here. We compute directly that the invariants are
spanned by $x_1x_3 + x_2x_3$, and so have dimension $1$.  Thus, $H^2(G,
k^\times) \cong H^2(G,\FF_3) \cong  C_3$.

Now suppose that $\D = \K$ or $\F$. Then $G=\Out_\D(P) \cong C_3^3 \rtimes 
(\gen{\mathbf{d}} \times S_3)$ with $\mathbf{d}$ inverting and we have \begin{align*}\label{e:wreathinvariants}
H^2(G,\mathbb{F}_3) = H^2(C_3 \wr C_3, \mathbb{F}_3)^{\langle \textbf{d}, \tau \rangle}, 
\end{align*} since a Sylow $3$-subgroup is normal in $G$. Let $W \cong C_3 \wr C_3$ be the
Sylow $3$-subgroup of $G$, and write $W_0$ for the base subgroup of $W$. 
By a result of Nakaoka \cite[Theorem~3.3]{Nakaoka1961}, we have

\begin{align*}
H^2(C_3 \wr C_3, \mathbb{F}_3) &= 
H^0(C_3 , H^2(W_0, \FF_3)) \oplus
H^1(C_3, H^1(W_0,\FF_3)) \oplus
H^2(C_3, H^0(W_0,\FF_3))
\end{align*}

The middle term above vanishes: by Lemma~\ref{lem:coh}(a), 
\[
H^1(W_0, \FF_3) \cong \Hom_{\FF_3}(\FF_3[C_3],\FF_3) = \Coind_1^{C_3} \FF_3
\]
as a $W/W_0$-module, so that $H^1(C_3, H^1(W_0, \FF_3)) \cong H^1(C_3,
\Coind_1^{C_3} \FF_3) = 0$ by Shapiro's Lemma \cite[Lemma~6.3.2]{WeibelHA}. 
Hence,
\[
H^2(C_3\wr C_3) = H^2(W_0, \FF_3)^{C_3} \oplus H^2(C_3, \FF_3). 
\]
With notation as in \eqref{e:wr}, the first summand is spanned by
$y_1+y_2+y_3$ and $x_1x_2+x_2x_3+x_3x_1$, both being negated by the action
of $\mathbf{d}\tau$ (note that $\tau$ negates $x_1x_2+x_2x_3+x_3x_1$).
Similarly, the second summand is also negated by $\mathbf{d}\tau$. Hence,
$H^2(G,\FF_3) = 0$, and we conclude that
$H^2(G,k^{\times}) = 0$ by Lemma~\ref{lem:coh}(c).
\mbox{}\newline \newline
\textit{Case 3.} $P = R_{1^7}$: Then $\D = \H$ or $\F$, and $\Out_\D(P) \cong A_7$.
The odd part of the Schur multiplier is well-known to be $C_3$. Alternatively,
apply Lemma~\ref{lem:coh}(h) to see that the exponent of the odd part of the
Schur multiplier is $3$, and then use Lemma~\ref{lem:coh}(e,g). 
\end{proof}

\begin{Thm}\label{t:lim0}
For $q$ an odd prime power and $\F = \Sol(q)$, we have
\[
\lim_{[S(\F^{cr})]} \A_\F^2 = 0.
\]
\end{Thm}
\begin{proof}
Let $(\F,\alpha)$ be a K\"ulshammer-Puig pair. When $l > 0$, all minimal
elements of the poset $[S(\F^{cr})]$, namely the chains $\sigma = (R)$ of
length one, have $\alpha_{[\sigma]} = 0$ by Lemma~\ref{l:trivmult}.  Thus, the
theorem holds in this case by Lemma~\ref{L:limposet}(a). 

It remains to consider the case $l = 0$. Then $H^2(\Out_\F(P), k^\times)$ is
nonzero (of order $3$) if and only if $P = R_{1^7}$ or $QR_{1^7}$.  For the
remainder of the proof, we set $R := R_{1^7}$, for short. Consider the chains
$\sigma := (R < QR)$ and $\tau := (Q < QR)$. All three subgroups $R$, $Q$, and
$QR$ are weakly $\F$-closed by Lemma \ref{l:qrweakly}; hence $\Aut_\F(\sigma) =
\Aut_\F(QR) = \Aut_\F(\tau)$ and the induced map on $\A^2$ is the identity in
each of these cases. We next prove that the induced map
\begin{equation}\label{e:commres}
H^2(\Aut_{\F}(R), k^{\times}) \to H^2(\Aut_\F(\sigma), k^{\times})
\end{equation}
is injective.
Once this is done, Lemma~\ref{L:limposet}(b) then yields that
$\lim_{[S(\F^{cr}])} \A^2 = 0$.

By Lemmas~\ref{l:qrweakly} and \ref{l:QRandQR*}, $QR$ contains $R$ as a
normal subgroup with index four, and $QR/R \cong C_2 \times
C_2$.  Hence, Lemma~\ref{l:abovewc} yields that the restriction map
$\Aut_{\F}(\sigma) = \Aut_\F(QR) \to \Aut_\F(R)$ induces an
isomorphism
\[
\Aut_\F(QR)/\Aut_{R}(QR) \longrightarrow N_{\Out_{\F}(R)}(\Out_{QR}(R)). 
\]
This isomorphism identifies
$\Aut_\F(QR)/\Aut_{R}(QR)$ with the normalizer in
$\Out_\F(R) \cong A_7$ of the four subgroup $QR/R
\cong \Aut_{QR}(R)/\Aut_{R}(R)$. 

Since this normalizer contains a Sylow $3$-subgroup of $A_7$, we conclude that
the restriction map 
\[
\rho_3\colon H^2(\Aut_{\F}(R), \FF_3) \longrightarrow H^2(\Aut_{\F}(\sigma), \FF_3)
\]
in $\FF_3$-cohomology is injective by \cite[Corollary~3.6.18]{Benson1998b}. 
By \cite[6.7.6]{WeibelHA} on the functoriality of restriction, the diagram 
\[
\xymatrix{
H^2(\Aut_\F(R), \FF_3) \ar[r]^{\rho_3}\ar[d] & H^2(\Aut_\F(\sigma), \FF_3) \ar[d] \\
H^2(\Aut_\F(R), k^\times)\otimes \mathbb{Z}_{(3)} \ar[r]^{\rho_{(3)}} & H^2(\Aut_\F(\sigma), k^\times)\otimes \mathbb{Z}_{(3)}\\
}
\]
commutes. Here, the vertical arrows are given by the isomorphisms of Lemma
\ref{lem:coh}(g), which applies since $\Aut_{\F}(R)$ and
$\Aut_{\F}(\sigma)$ have Sylow $3$-subgroups of order $3^2$. Therefore
$\rho_{(3)}$ is injective, as claimed. This completes the proof in the
case $l=0$ and of the theorem.
\end{proof}

\begin{proof}[Proof of Theorem \ref{t:solblocks}]
By \cite{LeviOliver2002} there exists a centric linking system associated with
$\F$. Thus \cite[Theorem 1.2]{Libman2011} yields that 
\[ 
\lim_{[S(\F^{c})]} \A_\F^2 \cong
\lim_{[S(\F^{cr})]} \A_\F^2.
\] 
The result now follows from Theorem \ref{t:lim0}.
\end{proof}

\section*{Appendix: Hasse diagrams}

Displayed below without proof are Hasse diagrams for the poset of isomorphism
classes of centric radicals in $\Sol(q)$ that were computed with the aid of
Magma \cite{Magma}. 

\begin{center}
\begin{figure}[h]
\label{hasse1}
\caption{Hasse diagram for $[\Sol(q)^{cr}]$, $q \equiv \pm 3 \pmod{8}$} 
\vspace{1cm}
\begin{tikzpicture}
    \node (4) at (-4,0) {$2^4$};
    \node (6) at (-4,2) {$2^6$};
    \node (7) at (-4,3) {$2^7$};
    \node (8) at (-4,4) {$2^8$};
    \node (9) at (-4,6) {$2^9$};
    \node (10) at (-4,8) {$2^{10}$};
    \node (S) at (2,8) {$S$};
    \node (CSU) at (7,6) {$C_S(U)$};
    \node (QR*) at (5,6) {$QR_{1^7}'$};
    \node (QR) at (3,6) {$QR_{1^7}$};
    \node (RR*) at (1,6) {$C_S(E/Z)$};
    \node (Q) at (5,4) {$Q$};
    \node (CSE) at (6,3) {$C_S(E)$};
    \node (R) at (1,3) {$R_{1^7}$};
    \node (R*) at (3,2) {$R_{1^7}'$};
    \node (A) at (3,0) {$A$};
    \draw (A)--(R);
    \draw (A)--(CSE);
    \draw (CSE)--(RR*);
    \draw (R*)--(RR*); 
    \draw (R*)--(QR*);
    \draw (Q)--(QR);
    \draw (Q)--(QR*);
    \draw (Q)--(CSU);
    \draw (R)--(RR*);
    \draw (R)--(QR);
    \draw (RR*)--(S);
    \draw (QR)--(S);
    \draw (QR*)--(S);
    \draw (CSU)--(S);
    \draw (CSE)--(CSU);
\end{tikzpicture}
\end{figure}
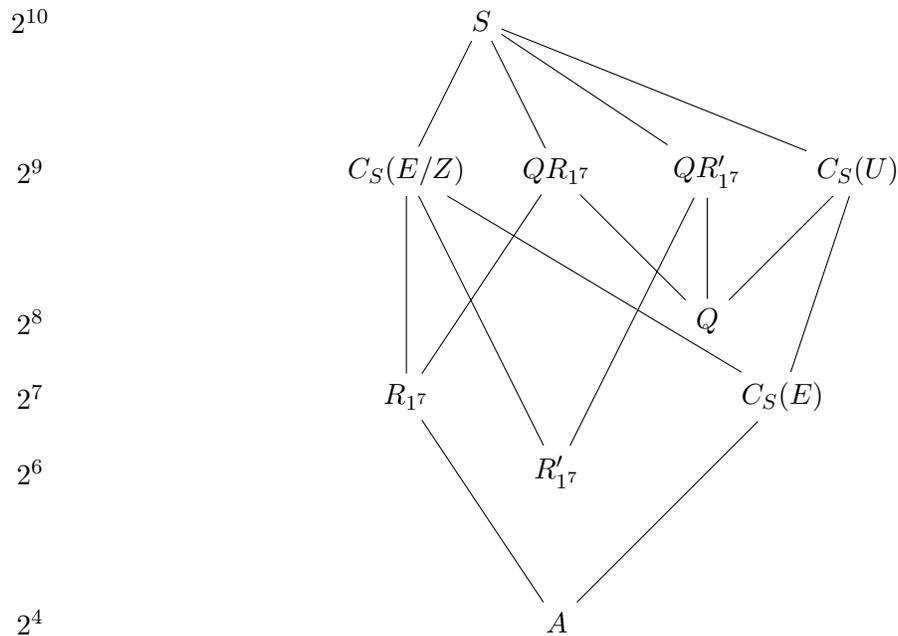
\end{center}

\begin{center}
\begin{figure}[h]
\label{hasse2}
\caption{Hasse diagram for $[\Sol(q)^{cr}]$, $q \equiv \pm 7 \pmod{16}$, i.e. for $l = 1$}
\vspace{1cm}
\begin{tikzpicture}
\node (24) at (-4,-10) {$2^4$};
 \node (26) at (-4,-9) {$2^6$};
    
  \node (27) at (-4,-8) {$2^7$};
    \node (28) at (-4,-7) {$2^8$};
  \node (29) at (-4,-6) {$2^9$};

\node (27l) at (-4,-3) {$2^{7+l}$};
   \node (28l) at (-4,-1) {$2^{8+l}$};
 \node (29l) at (-4,2) {$2^{9+l}$};
    \node (292l) at (-4,4) {$2^{9+2l}$};
\node (273l) at (-4,6) {$2^{7+3l}$};    
   
    \node (293l) at (-4,8) {$2^{9+3l}$};
    \node (2103l) at (-4,10) {$2^{10+3l}$};

    \node (1) at (6,10) {$S$};
   
    \node (2) at (-0.5,8) {$C_S(U)$};
    \node (11) at (0.5,4) {$R_1R_2Q_3\langle \tau \rangle$};

    \node (9) at (9.5,2) {$Q_1Q_2R_3\langle \tau \rangle$};
    \node (10) at (5,2) {$Q_1Q_2'R_3\langle \tau' \rangle$};

    \node (5) at (-1,-1) {$Q_1Q_2R_3$};
    \node (6) at (4.5,-1) {$Q_1Q_2'R_3$};
    \node (7) at (1,-6) {$Q_1Q_2Q_3\langle \tau \rangle$};
    \node (8) at (6,-6) {$Q_1Q_2Q_3'\langle \tau \rangle$};

      \node (3) at (-1,-7) {$Q_1Q_2Q_3$};
    \node (4) at (3,-7) {$Q_1Q_2Q_3'$};
    
    \node (12) at (11.5,-8) {$R_{1^7}$};
    \node (13) at (8,-9) {$R_{1^7}'$};
        \node (14) at (12,-3) {$R_{1^52}$};
        
   \node (15) at (12,8) {$C_S(E/Z)$};
  \node (16) at (10.5,-10) {$A$};
  \node (17) at (6.5,6) {$C_S(E)$};

    \draw (1)--(2);
  \draw (1)--(10);  
  \draw (1)--(11);
  \draw (1)--(9);
   \draw (1)--(15);

  \draw (2)--(5);
  \draw (2)--(6);
  \draw (2)--(17);
    
   \draw (3)--(5);
   \draw (3)--(7);
  
   \draw (4)--(5);
   \draw (4)--(6);
   \draw (4)--(8);
   \draw (4)--(11);
    \draw (5)--(9);

    \draw (6)--(10);

       \draw (7)--(9);
       \draw (7)--(11);
       \draw (7)--(12);    
    \draw (8)--(9);
    \draw (9)--(14);
    \draw (10)--(13);

      \draw (12)--(14); 
      \draw (12)--(16); 
  \draw (13)--(15);
    
       \draw (14)--(15); 
     \draw (15)--(17);
        \draw (16)--(17);
\end{tikzpicture}
\end{figure}
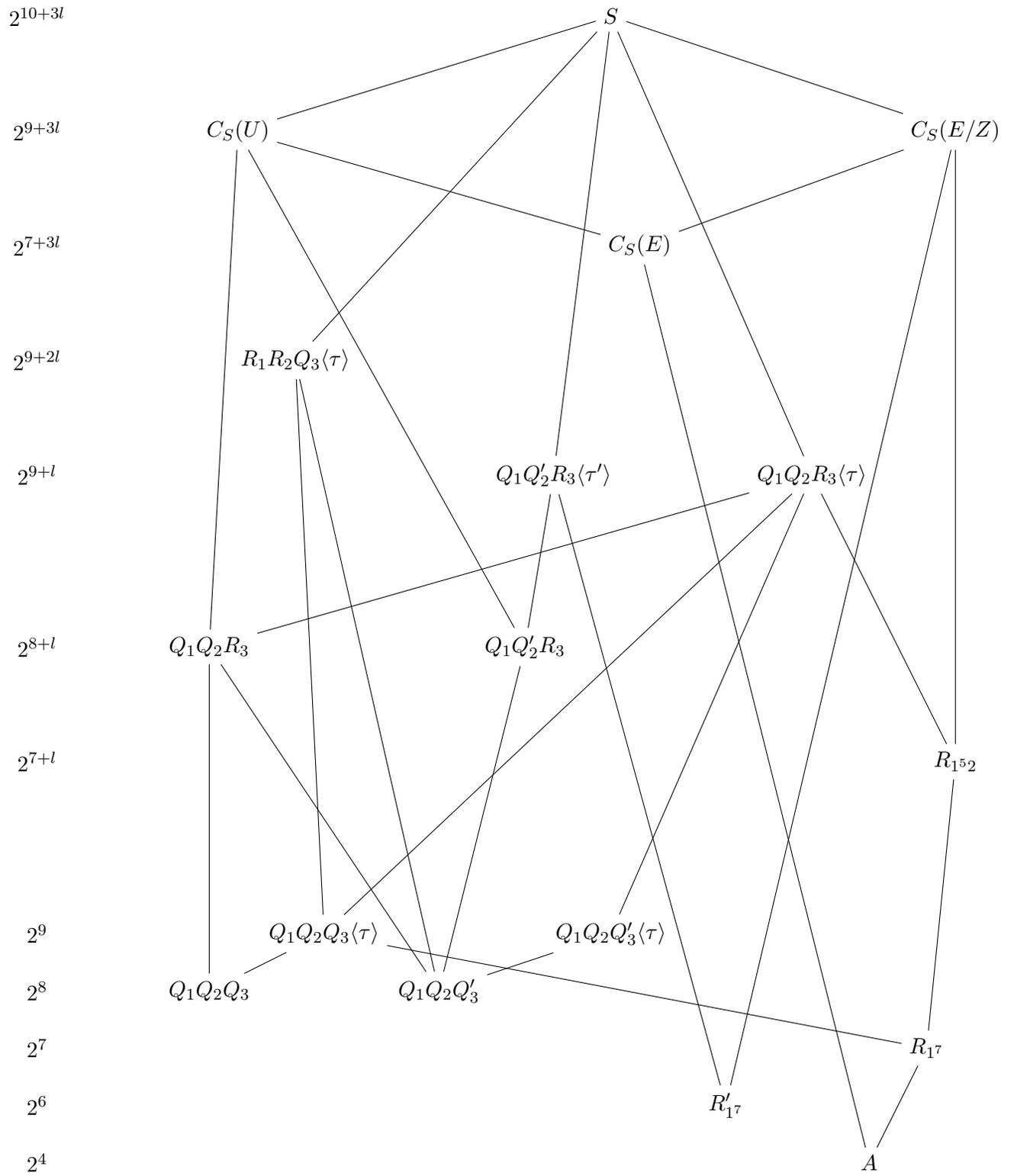
\end{center}

\bibliographystyle{amsalpha}{}
\bibliography{mybib}
\end{document}